\documentclass[]{article}

\renewcommand{\v}[1]{\ensuremath{\mathbf{#1}}}

\newcommand{\mc}{\mathcal}

\def\st{{\rm s.t.}}

\usepackage{algorithmic}
\usepackage{algorithm}
\usepackage{booktabs}
\usepackage[usenames,dvipsnames]{xcolor}
\usepackage{fullpage,graphicx,amssymb,amsmath}
\usepackage{enumerate}
\usepackage{caption,subcaption}
\usepackage{empheq}
\usepackage{url}

\usepackage{pdflscape}
\usepackage{afterpage}

\newcommand{\be}{\begin{enumerate}}
\newcommand{\ee}{\end{enumerate}}

\newtheorem{proposition}{Proposition}

\newcommand{\pf}{\textbf{Proof} \indent}
\newcommand{\qed}{\hfill $\Box$}

\newcommand{\proof}{\pf}

\renewcommand{\Re}{\mathbb{R}} 
\newcommand{\Z}{\mathbb{Z}} 

\newcommand{\vgamma}{\boldsymbol{\gamma}}

\newcommand{\vlambda}{\boldsymbol{\lambda}}

\newcommand{\vmu}{\boldsymbol{\mu}}
\newcommand{\vnu}{\boldsymbol{\nu}}

\newcommand{\vpi}{\boldsymbol{\pi}}

\newcommand{\IN}{\textrm{in}}
\newcommand{\OUT}{\textrm{out}}
\newcommand{\Nblend}{\mc{N}^{\textrm{pool}}}
\newcommand{\Npool}{\mc{N}^{\textrm{pool}}}
\newcommand{\Nin}{\mc{N}^{\IN}}
\newcommand{\Nout}{\mc{N}^{\OUT}}

\newcommand{\w}[1]{w_{1}#1,\dots,w_{|\mathcal{J}|}#1}
\newcommand{\x}[1]{\v{x}_{1}#1,\dots,\v{x}_{|\mathcal{J}|}#1}
\newcommand{\hatw}[1]{\hat{w}_{1}#1,\dots,\hat{w}_{|\mathcal{J}|}#1}
\newcommand{\hatx}[1]{\hat{\v{x}}_{1}#1,\dots,\hat{\v{x}}_{|\mathcal{J}|}#1}
\newcommand{\Players}{\mc{J}}

\newcommand{\cmt}[1]{#1}
\newcommand{\cmtt}[1]{#1}

\title{Pooling problems under perfect and imperfect competition}
\author{Dimitri J. Papageorgiou$^1$, Stuart M. Harwood$^1$, Francisco Trespalacios$^2$ \\
\\
{\small $^1$ExxonMobil Technology and Engineering Company -- Research}\\
{\small 1545 Route 22 East, Annandale, NJ 08801 USA}\\
{\small $^2$ExxonMobil Technology and Engineering Company -- Engineering}\\
{\small 22777 Springwoods Village Parkway, Spring, TX 77389 USA}\\
{\small \{dimitri.j.papageorgiou,stuart.m.harwood,francisco.trespalacios\}@exxonmobil.com} \\
}

\begin{document}

\maketitle

\begin{abstract}
We investigate pooling problems in which multiple players vie with one another to maximize individual profit in a non-cooperative competitive market. This competitive setting is interesting and worthy of study because the majority of prevailing process systems engineering models largely overlook the non-cooperative strategies that exist in real-world markets. In this work, each player controls a processing network involving intermediate tanks (or pools) where raw materials are blended together before being further combined into final products. Each player then solves a pure or mixed-integer bilinear optimization problem whose profit is influenced by other players. We present several bilevel formulations and numerical results of a novel decomposition algorithm. 

\vspace{2mm}
\noindent \textbf{keywords:} bilevel optimization, blending, equilibrium modeling, game theory, pooling. 
\end{abstract}

\section{Introduction} \label{sec:intro}

Pooling problems arise in myriad industrial applications and have consequently been the subject of numerous papers, surveys \cite{gupte2017relaxations,misener2009advances}, and theses \cite{alfaki2012models,gupte2012mixed,misener2012novel,pham2010global}.  
The prototypical pooling problem involves an acyclic network containing 
a set of input/supply nodes where raw materials enter the system, 
a set of intermediate nodes representing blending tanks (or \textit{pools}), and 
a set of output nodes where final products exit the system.  Directed arcs connect nodes and represent
possible flows.
There are typically constraints on the qualities or specifications of the streams flowing into various nodes, such as bounds on chemical and physical properties. 
Additional constraints on feedstock availability, storage capacity, demand, and product specifications are often included.
Thanks to their challenging nonconvex structures, they have also spurred noteworthy improvements in global and mixed-integer nonlinear programming (MINLP) optimization algorithms and solvers.  

Virtually all existing literature on the pooling problem has started from the vantage point of a single player (e.g., a planner, scheduler, or generic decision maker) who receives deterministic values for raw material costs, arc-flow costs, and prices for final goods.
In this work, we deviate from this traditional ``single-player" mindset and present a competitive pooling problem in an interactive market setting.
As a consequence, we spend little time identifying the best way to cast a particular pooling problem and instead turn our attention to modeling the dynamics of a multi-player non-cooperative game in which one player's decisions affect the prices seen by its competitors.  

This competitive setting is novel for the pooling problem and has strategic appeal as it seeks to represent the simultaneous optimization problems of multiple interacting decision makers.  In economics and energy systems modeling, this approach is commonly referred to as a bottom-up (or dis-aggregate) approach, meaning that the model attempts to capture the behavior of each market participant rather than assume aggregate market characteristics.  
This bottom-up modeling approach is interesting as it can be used in several ways.
First, it allows a market participant to optimize its decisions relative to those of other players, rather than assuming fixed demand.
Second, it can be used by policy makers attempting to, for example, impose a price on emissions to dis-incentivize players from producing certain final products. 

We first consider a perfect competition setting in which multiple players act as price takers who cannot anticipate the impact of their supply quantity on the market price. This model is of interest when many players are assumed, each with relatively small market power. 
We then investigate an imperfect competition setting when the suppliers are depicted as Nash-Cournot players competing against one another to maximize profit.  In a Nash-Cournot setting, each player acts as a ``price maker" or ``price setter" because they have additional information, namely, they can anticipate the market price of goods as a function of the quantities supplied by all players.  This setting is justified when there are relatively few players each wielding non-negligible market power.   

\subsection{Pooling problem literature review}

Pooling problems have been extensively studied in the process systems engineering and mathematical optimization literature as they are foundational in numerous applications, including petrochemical refining \cite{pham2010global}, wastewater treatment \cite{galan1998optimal}, mining \cite{boland2015special}, copper blending \cite{song2018scheduling}, and more. 
Misener and Floudas \cite{misener2009advances} attempt to categorize pooling problems into five broad classes based on the form of the nonconvexities that arise in the resulting optimization model:
\begin{enumerate}
\item The \textit{standard pooling problem} is a static (i.e., it has no time dimension) continuous bilinear NLP, where the complicating bilinearities arise due to a linear blending rule assumption for all streams that enter a pool.
\item The \textit{generalized pooling problem} is a mixed-integer bilinear program (MIBLP) in which inter-pool connections are permitted and may include discrete decisions related to network design choices.
\item The \textit{extended pooling problem}, introduced in Gounaris and Floudas \cite{gounaris2007formulation}, is a general mixed-integer nonlinear program (MINLP) that appends emissions constraints to the standard pooling problem.
\item The \textit{nonlinear blending problem} is nonconvex NLP that relaxes the assumption that all streams are blended linearly.
\item The \textit{crude oil operations problem} (or front-end crude scheduling problem as it applies to the front-end of a refinery) is a dynamic nonconvex MINLP.
\end{enumerate}
Note that the first four categories consider static problems, i.e., they have no time dimension. 
Meanwhile, dynamic or time-indexed models are often called for in scheduling applications where streams flow from tank to tank over multiple time periods.  This occurs in numerous refining applications. 

Gupte et al. \cite{gupte2015pooling} also offers a categorization with an emphasis on the type of formulation.  They highlight the ``concentration," ``proportion," and an augmented ``proportion" formulation, which are also known under their more cryptic names as the $\mathbb{P}$, $\mathbb{Q}$, and $\mathbb{PQ}$ formulations, respectively. Alfaki and Haugland \cite{alfaki2013strong} and Boland et al. \cite{Boland2016} proposed several multi-commodity flow formulations for the continuous version of the pooling problem.

Other important research efforts include Audet et al. \cite{audet2004pooling}, Misener and Floudas \cite{misener2009advances}, and Misener at al. \cite{misener2011apogee}.
More recently, Ceccon and Misener \cite{ceccon2021solving} attempt to solve large-scale pooling problems with an open-source solver GALINI.  Luedtke et al. \cite{luedtke2020strong} present convex nonlinear relaxations of the pooling problem to help achieve superior bounds.
Li et al. \cite{li2011stochastic} considered the design and operation of natural gas production networks in which input specification levels are treated as uncertain and handled within a two-stage stochastic programming model.

In all of these models, the cost or revenue of each output stream is assumed given or is specified as a deterministic parameter or through a probability distribution or uncertainty set.  In this work, we deviate from this assumption by assuming that the unit revenue for each commodity is determined in a competitive market consisting of multiple players each solving their own pooling problem. In a perfect competition setting, each producer is a price taker where the price is a market clearing price.  In an imperfect Nash-Cournot competition setting, each producer knows the market's inverse demand function and attempts to maximize profit given its knowledge about this function and its perception of how other producers will behave.

\subsection{Modeling competitive markets}

We suppose throughout that there are $N$ players indexed by $j \in \mc{J}=\{1,\dots,N\}$, each trying to optimize their individual decisions. In particular, suppose player $j$ controls the vector $\v{x}_j \in \Re^{n_j}$ of $n_j$ decision variables and is trying to maximize her profit function $p_j: \Re^{n_j} \times \Re^{d} \mapsto \Re$ that depends on her decision $\v{x}_j$ and a vector $\vpi \in \Re^d$.
Then, player $j$'s optimization problem can be stated as    
\begin{equation} \label{model:intro_agent_problem_Nash}
p_{j}^*(\vpi) = \max \left\{p_j(\v{x}_j,\vpi): \v{x}_j \in \mc{X}_j \right\}.
\end{equation}
Correspondingly, let $\mc{S}_j(\vpi)$ denote player $j$'s set of optimal solutions given $\vpi$, i.e., 
\begin{equation}
\mc{S}_j(\vpi) = \arg \max \left\{p_j(\v{x}_j,\vpi): \v{x}_j \in \mc{X}_j \right\}.
\end{equation}
In addition, we assume that $\vpi$ and the players' decisions are linked via constraints in the set
$\Pi \subseteq \Re^d \times \Re^n$, where $n = \sum_{j \in \mc{J}} n_j$.
Then, an equilibrium of this multi-player game is a point 
$(\vpi,\v{x})=(\vpi,\v{x}_1,\dots,\v{x}_N)$
such that $(\vpi,\v{x}) \in \Pi$ and $\v{x}_j \in \mc{S}_j(\vpi)$ for all players $j \in \mc{J}$.

Note that this setup includes the special case when $\vpi$ simply comprises all players' decisions, i.e., when
$\Pi = \{ (\vpi_1,\dots,\vpi_N,\v{x}_1,\dots,\v{x}_N) \in \Re^{2n} : \vpi_j = \v{x}_j ~\forall j \in \mc{J} \}$.
In this special case, one could just as well forgo introducing $\vpi$ and define player $j$'s optimal decision set as
\begin{equation*}
\tilde{\mc{S}}_j(\v{x}_{-j}) = \arg \max \left\{\tilde{p}_j(\v{x}_j,\v{x}_{-j}): \v{x}_j \in \mc{X}_j \right\},
\end{equation*} 
where a tilde is used to differentiate from our original notation.
In this special case, a Nash equilibrium is a point $\v{x}=(\v{x}_1,\dots,\v{x}_N)$ such that $\v{x}_j \in \tilde{\mc{S}}_j(\v{x}_{-j})$ for all players $j \in \mc{J}$.
As will become clear, we prefer to include $\vpi$ to allow for external constraints and more complex interactions between the players. 
A proof of the equivalence between these two representations is given in \cite{harwood2021equilibrium}.

With this setup established, we now describe some relevant previous work in competitive games and modeling approaches thereof. 
Ruiz et al. \cite{ruiz2014tutorial} present a number of interesting market equilibrium models germane to the energy sector. The thrust of their tutorial is on how  complementarity models can be used to model the interactions between players.  While a complementarity approach is a propos when each player is optimizing over ``nice" sets 
as they require the existence of Karush-Kuhn-Tucker (KKT) optimality conditions, they are limiting and/or insufficient when more complicated lower-level feasible regions are present. For example, if players possess feasible regions defined by non-differentiable functions, then KKT conditions are neither necessary nor sufficient and are therefore inadequate. In the context of pooling problems, if complex EPA emissions constraints (which include nonconvex, non-smooth functions) are included in a player's optimization model, then complementarity models are likely not appropriate. 

In economics, a common approach to finding a Nash equilibrium to this multi-player game is to rely on KKT optimality conditions to derive a complementarity system. Specifically, the economics community typically assumes that each player's feasible region can be expressed as $\mc{X}_j = \left\{ \v{x}_j \in \Re^{n_j}: \v{g}_j(\v{x}_j) \leq \v{0}, \v{h}_j(\v{x}_j) = \v{0}, \v{x}_j \geq \v{0} \right\}$ where the vector-valued functions $\v{g}_j: \Re^{n_j} \mapsto \Re^{m_j}$ and $\v{h}_j: \Re^{n_j} \mapsto \Re^{k_j}$ are continuously differentiable. Under these assumptions, the KKT conditions for \eqref{model:intro_agent_problem_Nash} become 
\begin{subequations} \label{model:generic_nlp_kkt}
\begin{alignat}{4}
\nabla_{\v{x}_j} L_j = \nabla_{\v{x}_j} p_j(\v{x}_j,\vpi) 
+ \vlambda_j^{\top}\nabla_{\v{x}_j} \v{g}_j(\v{x}_j) 
+ \vmu_j^{\top}\nabla_{\v{x}_j} \v{h}_j(\v{x}_j)
- \vnu_j = \v{0} & & \label{eq:partial_Lagrangian} \\ 
\v{0} \leq -\v{g}_j(\v{x}_j) \perp \vlambda_j \geq \v{0} & & \label{eq:gj_perp_lambdaj} \\
\v{0} = \v{h}_j(\v{x}_j)~,~\vmu_j \in \Re^{k_j} & & \label{eq:hj_perp_muj} \\
\v{0} \leq \v{x}_j \perp \vnu_j \geq \v{0} & &~ \label{eq:xj_perp_nuj},
\end{alignat}
\end{subequations}
where $L_j$ is the Lagrangian function for player $j$ and the ``perp" symbol $\perp$ is shorthand for complementarity. For example, $\v{0} \leq -\v{g}_j(\v{x}_j) \perp \vlambda_j \geq \v{0}$ is a compact way of expressing the three conditions $\v{g}_j(\v{x}_j) \leq \v{0}$, $\vlambda_j \geq \v{0}$, and $\v{g}_j(\v{x}_j)^{\top}\vlambda_j = \v{0}$.  
Of course, these conditions are only valid if $p_j$, $\v{g}_j$, and $\v{h}_j$ are continuously differentiable, which is not necessarily the case in the generalized pooling problem setting.
Moreover, even if $p_j$, $\v{g}_j$, and $\v{h}_j$ are continuously differentiable, they may not be convex, in which case these conditions are necessary, but not sufficient.  Indeed, the pooling problems considered in this work have nonconvex feasible regions. 

To handle discretely-constrained equilibrium problems with convex players akin to those stated above, i.e., the continuous relaxation of each player's problem is convex, Gabriel et al. \cite{Gabriel2013} offer a heuristic that provides a compromise between complementarity and integrality.  The approach works ``by first relaxing the discretely-constrained variables to their continuous analogs, taking KKT conditions for this relaxed problem, converting these conditions to disjunctive-constraints form (Fortuny-Amat and McCarl 1981), and then solving them along with the original integer restrictions re-inserted in a mixed-integer, linear program (MILP). The integer conditions are then further relaxed, but targeted using penalty terms in the objective function. This MILP relaxes both complementarity and integrality but tries to find minimum deviations for both" \cite[p.308]{Gabriel2013}. Fuller and \c{C}elebi \cite{Fuller2017436} suggest another approach for handling discrete decisions.

Equilibrium problems can also be modeled as bilevel optimization problems.
For this work, we adapt the approach of Harwood et al. \cite{harwood2021equilibrium}, which introduces an exact method for equilibrium problems with nonconvex structures. 
Meanwhile, two commonly used bilevel models are: (1) Mathematical programs with equilibrium constraints (MPECs), which are optimization problems whose constraints include equilibrium conditions (e.g., KKT conditions of a lower-level optimization problem or Nash-Cournot game). (2) Equilibrium problems with equilibrium constraints (EPECs), which typically arise when there are multiple leaders in a Stackelberg game.  We do not pursue these objects in this work, as they suffer from the same drawbacks as the complementarity approaches when players are nonconvex.  
Colson et al. \cite{colson2007overview} provide an overview of bilevel optimization with some connections to equilibrium modeling.

\subsection{Contributions}
The contributions of this paper are:
\begin{enumerate}
\item We introduce a novel pooling application that couples the classic pooling problem with a non-cooperative equilibrium modeling framework in which multiple producers compete to maximize their individual profit.
\item We present perfect and imperfect competition models and formulate them as bilevel feasibility problems with lower-level nonconvex structures.
\item We apply a provably optimal bilevel decomposition algorithm to minimize disequilibrium. We showcase how our minimum disequilibrium approach is able to rigorously and systematically determine that no equilibrium exists, when this is the case, while returning the total disequilibrium (a strictly positive scalar) as a certificate of non-existence.  We compare this against a complementarity heuristic, which not only fails to provide a certificate of non-existence, but worse, fails to identify that no equilibrium exists.
\item We highlight the nuances that result when contrasting a traditional deterministic approach, a perfect competition, and Nash-Cournot setting, with Haverly-esque pooling players.
\end{enumerate}

It is worth mentioning what is not considered in this paper. First, we restrict our attention to a static setting (no time dimension) despite the fact the models introduced are applicable for dynamic (time-indexed) problems as well, albeit with some care for inventory management and valuation.  Second, given the expansive literature on formulating and solving an individual pooling problem, we make no attempt whatsoever to improve upon these formulations. In fact, we rely on a standard $\mathbb{P}$-formulation in our presentation and numerical studies. Note, however, that all of our competitive models carry over to other popular pooling problem formulations. 
We only consider \textit{linear} inverse demand functions as they offer a practically important first step \cite{ruiz2014tutorial}.

The remainder of this paper is organized as follows:
Section~\ref{sec:competitive_pooling_models} introduces the generalized pooling problem considered for each player as well as the perfect and imperfect competition settings that we study. 
Section~\ref{sec:review} reviews some fundamental concepts of bilevel optimization and complementarity that we employ for finding equilibria. Special attention is given to the difficulties associated with complementarity approaches when nonconvexities are present.
Section~\ref{sec:algorithms} discusses algorithms for solving competitive pooling problems heuristically and to global optimality.
Section~\ref{sec:results} presents numerical results on various instances.
Conclusions and future research directions are highlighted in Section~\ref{sec:conclusions}.

\newpage
\section{Competitive pooling models} \label{sec:competitive_pooling_models}

\subsection{Nomenclature}
\begin{tabular}{ll}
\toprule
& \textbf{Definition} \\
\hline
\textbf{Sets} & \\
$i,j \in \mc{J}$ & set of players; $\mc{J} = \{1,\dots,N\}$ \\
$k \in \mc{K}$ & set of specifications (specs) or qualities \\
$l \in \mc{L}_j$ & player $j$'s set of discrete decisions \\
$n \in \Nin_j$ & set of input nodes where raw materials (input streams) enter player $j$'s network \\
$n \in \Nout_j$ & set of output nodes where finished goods (output streams) exit player $j$'s network \\
$n \in \Npool_j$ & set of intermediate pool nodes (\textit{pools}) in which streams are pooled together and blended \\
$n \in \mc{N}_j$ & set of nodes in player $j$'s network \\
$(n,n') \in \mc{A}_j$ & set of arcs (node-to-node connections) in player $j$'s network \\
$n \in \Nout$ & set of all output nodes; $\Nout = \cup_{j \in \mc{J}} \Nout_j$ \\
$\mc{X}_j ~(\mc{X}^C_j)$ & player $j$'s feasible region (the continuous relaxation of player $j$'s feasible region) \\
$\Pi$ & global feasibility set linking $\vpi$ and individual player decisions $\v{x}_j$; $(\vpi,\x{}) \in \Pi$ \\
\textbf{Parameters} &  \\
$\alpha_{n}$~$(\beta_{n})$ & inverse demand function intercept (slope) for the output stream associated with node $n \in \Nout$ \\
$\eta$~/~$\eta^L$~$(\eta^U)$ & minimum disequilibrium / lower (upper) bound on minimum disequilibrium \\
$C_{jn}^{\textrm{var}}$  & variable cost for player $j$ to purchase one unit of raw material $n \in \Nin_j$ \\
$C_{jn}^{\textrm{fixed}}$  & fixed cost for player $j$ to purchase any positive amount of raw material $n \in \Nin_j$ \\
$C_{nk}^{\IN}$  & concentration of spec $k$ in the raw material entering at player $j$'s node $n \in \Nin_j$ \\
$C_{nk}^{\min}$ ($C_{nk}^{\max}$) & minimum (maximum) permissible concentration of spec $k$ in a finished good at output node $n \in \Nout$ \\
$F_{jn}^{\min}$ ($F_{jn}^{\max}$) & minimum (maximum) amount of flow through player $j$'s node $n \in \mc{N}$ \\
\textbf{Variables} & \\
$c_{jnk}$ & (continuous) concentration of player $j$'s spec $k$ in output flow at blend node $n \in \Nblend$ \\
$f_{jn}^{\IN}$ & (continuous) flow into (amount purchased at) player $j$' input node $n \in \Nin_j$ \\
$f_{jn}^{\OUT}$ & (continuous) flow out of (amount sold at) player $j$'s output node $n \in \Nout_j$ \\
$f_{jnn'}$ & (continuous) flow on player $j$'s arc $(n,n') \in \mc{A}_j$ \\
$q_{n}$ & (continuous) quantity consumed of the output stream associated with node $n \in \Nout$ \\
$u_{jn}$ & (binary) takes value 1 if there is flow into player $j$'s input node $n \in \Nin_j$; 0 otherwise \\
$\v{x}_{j} \in \Re^{n_j}$ & generic decision vector for player $j \in \mc{J}$ \\
$\v{x}_{-j} \in \Re^{n-n_j}$ & generic vector of decisions for all players except player $j \in \mc{J}$ \\
$\vpi \in \Re^d$ & upper-level decision vector (e.g., price vector), sometimes treated as a parameter \\
\textbf{Functions} & \\
$\v{g}_j(\v{x}_j)$ & player $j$'s inequality constraints; $\v{g}_j: \Re^{n_j} \mapsto \Re^{m_j}$ \\
$\v{h}_j(\v{x}_j)$ & player $j$'s equality constraints; $\v{h}_j: \Re^{n_j} \mapsto \Re^{k_j}$ \\
$L_j$ & Lagrangian function associated with player $j$ \\
$p_j(\v{x}_j,\vpi)$ & player $j$'s profit function; $p_j: \Re^{n_j} \times \Re^{d} \mapsto \Re$ \\
$p_j^*(\vpi)$ & player $j$'s optimal objective function value (profit); $p_j^*: \Re^{d} \mapsto \Re$ \\
$p_j^U(\vpi)$ & upper (dual) bound on player $j$'s optimal objective function value (profit); $p_j^U: \Re^{d} \mapsto \Re$ \\
$\tilde{p}_j(\v{x}_j,\v{x}_{-j})$ & player $j$'s profit function in a traditional Nash-Cournot setting; $\tilde{p}_j: \Re^{n_j} \times \Re^{n-n_j} \mapsto \Re$ \\
$\tilde{p}_j^*(\v{x}_{-j})$ & player $j$'s optimal objective function value in a traditional Nash-Cournot setting; $\tilde{p}_j^*: \Re^{n-n_j} \mapsto \Re$ \\
$\tilde{p}_j^U(\v{x}_{-j})$ & upper bound on player $j$'s optimal objective function value (Nash-Cournot setting); $\tilde{p}_j^U: \Re^{n-n_j} \mapsto \Re$ \\
$\delta_j(\v{x}_j,\vpi)$ & player $j$'s disequilibrium; $\delta_j(\v{x}_j,\vpi) = p_{j}^*(\vpi) - p_j(\v{x}_j,\vpi)$ \\
\bottomrule
\end{tabular}
{\scriptsize Note on notational overloading: We admit to indexing nodes as $n \in \mc{N}$ and using $n = \sum_j n_j$ as the dimension of all player decisions.}

As stated in Section~\ref{sec:intro}, there are numerous formulations for the pooling problem and its generalization.  Below we use a variant of the standard $\mathbb{P}$-formulation. One could also use a $\mathbb{Q}$-, $\mathbb{PQ}$-, or source-based formulation \cite{lotero201613}. 
We assume that player $j$'s feasible region $\mc{X}_j$ 
is given by the following set of constraints (where the subscript $j$ is omitted for readability):
\begin{subequations} \label{model:core_pooling_problem}
\begin{alignat}{4}
& f_{n}^{\IN} = \sum_{(n,n') \in \mc{A}} f_{nn'} \quad \forall n \in \Nin & & \label{eq:fin_auxiliary} \\
& f_{n}^{\OUT} = \sum_{(n',n) \in \mc{A}} f_{n'n} \quad \forall n \in \Nout & & \label{eq:fout_auxiliary} \\
& \sum_{(n',n) \in \mc{A}} f_{n'n} = \sum_{(n,n') \in \mc{A}} f_{nn'} \quad \forall n \in \Nblend \qquad \textrm{[flow balance]} & & \label{eq:haverly_flow_balance_core} \\
& \sum_{\substack{n' \in \Nin:\\(n',n) \in \mc{A}}} C_{n'k}^{\IN} f_{n'n} + \sum_{\substack{n' \in \Nblend:\\(n',n) \in \mc{A}}} c_{n'k} f_{n'n} = c_{nk} \sum_{(n,n') \in \mc{A}} f_{nn'} \quad \forall n \in \Nblend, k \in \mc{K} \quad \textrm{[concentration balance]} & & \label{eq:concentration_mass_balance_core} \\
& \sum_{\substack{n' \in \Nin:\\(n',n) \in \mc{A}}} C_{n'k}^{\IN} f_{n'n} + \sum_{\substack{n' \in \Nblend:\\(n',n) \in \mc{A}}} c_{n'k} f_{n'n} \leq C_{nk}^{\max} f_{n}^{\OUT} \quad \forall n \in \Nout, k \in \mc{K}  \qquad \textrm{[max concentration]} & & \label{eq:max_concentration_core} \\
& \sum_{\substack{n' \in \Nin:\\(n',n) \in \mc{A}}} C_{n'k}^{\IN} f_{n'n} + \sum_{\substack{n' \in \Nblend:\\(n',n) \in \mc{A}}} c_{n'k} f_{n'n} \geq C_{nk}^{\min} f_{n}^{\OUT} \quad \forall n \in \Nout, k \in \mc{K}  \qquad \textrm{[min concentration]} & & \label{eq:min_concentration_core} \\
& F_n^{\min} u_{n} \leq f_{n}^{\IN} \leq F_n^{\max} u_{n}\quad \forall n \in \Nin & & \label{eq:semi_continuous_flow} \\
& c_{nk} \in [0,1] \quad \forall n \in \Nblend, k \in \mc{K} & & \\
& f_{nn'} \geq 0 \quad \forall (n,n') \in \mc{A} & & \\
& 0 \leq f_{n}^{\OUT} \leq F_n^{\max} \quad \forall n \in \Nout & & \\
& u_{n} \in \{0,1\} \quad \forall n \in \Nin 
\end{alignat}
\end{subequations}
Constraints~\eqref{eq:fin_auxiliary} and \eqref{eq:fout_auxiliary} track flow at each input and output node, respectively.  Note that the decision variables $f_{n}^{\IN}$ and $f_{n}^{\OUT}$ are redundant, but are stated here to improve readability.  
Constraints~\eqref{eq:haverly_flow_balance_core} ensure flow balance at all pools.
Constraints~\eqref{eq:concentration_mass_balance_core} ensure concentration balance of each spec at each pool. 
Constraints~\eqref{eq:max_concentration_core} and \eqref{eq:min_concentration_core} limit the concentration of certain specs at output nodes, e.g., a sulfur concentration limit for a diesel fuel.
Bilinear terms, which lead to nonconvex quadratic constraints, involving the product of two continuous decision variables (flow and concentration) appear in constraints~\eqref{eq:concentration_mass_balance_core}, \eqref{eq:max_concentration_core}, and \eqref{eq:min_concentration_core}.
Constraints~\eqref{eq:semi_continuous_flow} impose a semi-continuous nature on certain input flows by insisting that $f_{n}^{\IN} > 0 \implies f_{n}^{\IN} \in [F_n^{\min},F_n^{\max}]$; otherwise, $f_{n}^{\IN} = 0$.  This semi-continuous constraint is not found in the standard pooling problem, but may arise in a generalized pooling problem.  
The remaining constraints capture variable bounds.
In all, we have a set of linear and nonconvex quadratic constraints involving discrete and continuous variables, which give rise to challenging nonconvex QCQP and/or MIQCQP optimization problems.

\subsection{Perfect competition} \label{sec:pc_objectives}

In economics, a \textit{price taker} refers to a market participant who does not anticipate any effect of its decisions on the prices of goods or services and, therefore, must accept the prevailing market price. 
This impotency stems from the assumption that there are many competitors selling identical products and buyers have access to the price charged by all participants.

Assuming output streams are independent, player $j$'s objective function is:
\begin{equation} \label{objfnc:price_taker_indep}
\sum_{n \in \Nout_j} \pi_{n} f_{jn}^{\OUT} 
- \sum_{n \in \Nin_j} \left( C_{jn}^{\textrm{var}} f_{jn}^{\IN} + C_{jn}^{\textrm{fixed}} u_{jn} \right) 
\end{equation}
where the first summation captures player $j$'s revenue as the market-clearing price $\pi_n$ of the output stream at node $n \in \Nout_j$ times the flow out of that node. The second summation denotes the variable and fixed cost incurred. 
Meanwhile, we assume there is a single consumer (``consumer player'') who participates in two independent markets each governed by a linear inverse demand function defined by an intercept $\alpha_n$ and a negative slope $\beta_n$, whose goal is to maximize consumer surplus by solving
\begin{equation} \label{objfnc:consumer_pc}
\max \left\{ \sum_{n \in \Nout}(\alpha_{n} q_{n} - \tfrac{1}{2} \beta_{n} q_{n}^2 - \pi_{n} q_{n}) : q_{n} \geq 0, \forall n \in \Nout  \right\}.
\end{equation} 

\subsection{Nash-Cournot imperfect competition}

Nash-Cournot competition is a common economic paradigm used to model an industry structure in which suppliers compete via their production quantities. We follow the typical assumptions for this type of game-theoretic model that render it a so-called \textit{static non-cooperative game with complete information} \cite{gibbons1992game}: 
(i) suppliers are economically \textit{rational} and act strategically to maximize profit given their competitors' decisions; 
(ii) suppliers decide quantities \textit{simultaneously} (i.e., this is a static game); 
(iii) suppliers decide independently of one another and do not collude (hence, the term ``non-cooperative");
(iv) suppliers have \textit{complete information} about how their production quantities, and those of the competitors, affect the market clearing price; 
(v) suppliers have market power, i.e., each supplier's production quantity affects the market clearing price;
(vi) the number of suppliers is fixed.

Assuming output streams are independent and governed by linear inverse demand functions, player $j$'s objective function is:
\begin{equation} \label{objfnc:nash_cournot_indep_linear}
\sum_{n \in \Nout_j} \left[ \alpha_n - \beta_n \left( f_{jn}^{\OUT} + \sum_{i \in \mc{J}: i \neq j} f_{in}^{\OUT} \right) \right] f_{jn}^{\OUT} - \sum_{n \in \Nin_j} \left( C_{jn}^{\textrm{var}} f_{jn}^{\IN} + C_{jn}^{\textrm{fixed}} u_{jn} \right) 
\end{equation}

\section{Brief review of mathematical formulations for equilibrium modeling} \label{sec:review}

This section first outlines the concepts from bilevel optimization underpinning our algorithmic approach in the subsequent section. Namely, it describes the connections between bilevel optimization, semi-infinite programming, and minimizing disequilibrium.  
We then describe the maximum social welfare and complementarity models that we use as a basis of comparison. 
The last subsection offers a word of caution regarding complementarity approaches in a nonconvex setting.

\subsection{Bilevel optimization}

Consider a bilevel programming problem:
\begin{alignat}{2}
\label{blp}
\notag
\min_{\vpi,\v{x}}\; & \phi(\vpi,\v{x}) \\
\st~~
\tag{BLP} & (\vpi,\v{x}) \in \Pi \\
\notag & 
	\begin{aligned}
	\v{x} \in \arg 
	\max_{\v{y} \in \mc{X}}\; & p(\v{y},\vpi).
	\end{aligned}
\end{alignat}
%
We call $\vpi$ the upper (sometimes outer or leader) variables, and 
$\v{x}$ the lower (or inner or follower) variables.
The lower level problem is an optimization problem parameterized by $\vpi$:
\begin{equation} \label{model:lower_level_problem}
p^*(\vpi) = \max_{\v{y}} \left\{ p(\v{y},\vpi): \v{y} \in \mc{X} \right\},
\end{equation}
where we use the dummy vector $\v{y}$ to avoid confusion.
A valid reformulation of \eqref{blp} is the following:
\begin{alignat}{2} \label{blp_r}
\notag
\min_{\vpi,\v{x}}\; & \phi(\vpi,\v{x}) \\
\st~~
\tag{BLP-R} & (\vpi,\v{x}) \in \Pi \\
\notag & \v{x} \in \mc{X} \\
\label{con:infinite}	& p(\v{x},\vpi) \ge p(\v{y},\vpi), \forall \v{y} \in \mc{X}.
\end{alignat}
This reformulation has been considered in \cite{bard83,mitsos10,tsoukalasEA09,tuyEA93} and we recognize it as a semi-infinite program.
This class of problems and their numerical solution methods informs the approach from \cite{harwood2021equilibrium} for finding equilibrium.
In this work, we are particularly interested in the special case when the upper-level objective function $\phi(\vpi,\v{x})$ is trivial and everywhere equal to zero, in which case \eqref{blp_r} reduces to a feasibility problem, albeit a challenging one given the presence of ``infinite'' constraints \eqref{con:infinite}.  Harwood et al. \cite{harwood2021equilibrium} offers an attractive alternative to this feasibility problem by replacing the trivial objective with one that seeks to minimize the violation of the infinite constraint \eqref{con:infinite}.  While there are many options for measuring the violation, or in our context the ``amount of disequilibrium,'' we use arguably the simplest choice of minimizing the absolute (L1) distance between $p^*(\vpi)$ and $p(\v{x},\vpi)$, leading to the following model:
\begin{alignat}{2} \label{blp_MinDiseq}
\notag
\min_{\vpi,\v{x}}\; & p^*(\vpi) - p(\v{x},\vpi) \\
\st~~
\tag{BLP-R0} & (\vpi,\v{x}) \in \Pi \\
\notag & \v{x} \in \mc{X} 
\end{alignat}
The validity of this approach should be clear.  
By definition \eqref{model:lower_level_problem}, $p^*(\vpi) \geq p(\v{x},\vpi)$ for all $\v{x} \in \mc{X}$.
It follows that for any $(\vpi,\v{x})$ feasible to \eqref{blp_MinDiseq}, the objective value is greater than or equal to zero.
If, further, $(\vpi,\v{x})$ is optimal for \eqref{blp_MinDiseq} with optimal objective value equal to zero, then $(\vpi,\v{x})$ must also be feasible to \eqref{blp_r} and therefore optimal when the function $\phi$ is trivial and everywhere equal to zero.
We explore algorithms to solve \eqref{blp_MinDiseq} in Section~\ref{sec:algorithms}.

\subsection{Maximum social welfare for perfect competition} \label{sec:perfect_competition}

One is often interested in maximizing social welfare, the sum of producer and consumer surplus (i.e., the sum of the objective functions in \eqref{objfnc:price_taker_indep} and \eqref{objfnc:consumer_pc}), to determine static long-run general equilibrium conditions of a perfectly competitive game.
Under the assumptions of Section~\ref{sec:pc_objectives}, and after requiring the quantity consumed $q_n$ of output stream $n \in \Nout$ to equal the quantity produced $\sum_{j \in \mc{J}} f_{jn}^{\OUT}$, the maximum social welfare problem becomes a bilinearly-constrained optimization problem with a concave quadratic objective function (an MIQCQP):
\begin{alignat}{3} \label{model:max_social_welfare}
\max_{\v{q},\v{c},\v{f},\v{u}}~ & \sum_{n \in \Nout}(\alpha_{n} q_{n} - \tfrac{1}{2} \beta_{n} q_{n}^2) - \sum_{n \in \Nin_j} \left( C_{jn}^{\textrm{var}} f_{jn}^{\IN} + C_{jn}^{\textrm{fixed}} u_{jn} \right) \notag \\
\st~~ 				
	   & (\v{c}_j,\v{f}_j,\v{u}_j) \in \mc{X}_j 	\qquad \forall j \in \mc{J} \\
\notag & q_n = \sum_{j \in \mc{J}} f_{jn}^{\OUT}    \qquad \forall n \in \Nout
\end{alignat}
If desired, one can project \eqref{model:max_social_welfare} on to a lower-dimensional space by substituting $\sum_{j \in \mc{J}} f_{jn}^{\OUT}$ for $q_n$ to obtain
\begin{alignat}{3} \label{model:max_social_welfare_projected}
\max_{\v{c},\v{f},\v{u}}~ & \sum_{n \in \Nout} \left[\alpha_{n} \left( \sum_{j \in \mc{J}} f_{jn}^{\OUT} \right) - \tfrac{1}{2} \beta_{n} \left( \sum_{j \in \mc{J}} f_{jn}^{\OUT} \right)^2 \right] - \sum_{n \in \Nin_j} \left( C_{jn}^{\textrm{var}} f_{jn}^{\IN} + C_{jn}^{\textrm{fixed}} u_{jn} \right) \notag \\
\st~~ 				
	   & (\v{c}_j,\v{f}_j,\v{u}_j) \in \mc{X}_j 	\qquad \forall j \in \mc{J}
\end{alignat}

When all players are assumed to be price-takers, one can solve the maximum social welfare model \eqref{model:max_social_welfare} (or \eqref{model:max_social_welfare_projected}) to try to obtain an equilibrium.  
Specifically, Harwood et al. \cite{harwood2021equilibrium} proved that, when an equilibrium $(\vpi,\v{x})$ exists, the component $\v{x}=(\v{x}_1,\dots,\v{x}_N)$ corresponding to the individual player decisions, i.e., $\v{x}_j = (\v{c}_j,\v{f}_j,\v{u}_j)$ for all $j \in \mc{J}$ (pooling players, not suppliers), can be obtained by solving the maximum welfare problem regardless of convexity. In our pooling setting, when an equilibrium exists, prices can be obtained by setting $\pi_n = \alpha_n - \beta_n \sum_{j \in \mc{J}} f_{jn}^{\OUT}$ for all $n \in \Nout$ since the consumer model~\eqref{objfnc:consumer_pc} is quite simple.
More generally, when the individual player problems are convex, the equilibrium prices $\vpi$ can be obtained from the dual variables of the maximum social welfare problem. However, when the model is nonconvex, one must verify that strong duality holds in order to guarantee that optimal dual variables provide the equilibrium prices. Finding optimal dual variables, and subsequently verifying  whether strong duality holds, can be nontrivial for nonconvex problems in general. Because of this, when the model is nonconvex, existence of equilibrium cannot be proved by simply solving maximum welfare and using the resulting Lagrange multipliers to solve for each individual player's decisions in \eqref{model:intro_agent_problem_Nash}. 
When no equilibrium exists, the maximum social welfare model \eqref{model:max_social_welfare} alone is neither guaranteed to detect this non-existence, nor quantify the degree of disequilibrium.

For sake of completeness, Appendix~\ref{app:msw_via_comp} provides a standalone derivation of the maximum social welfare model \eqref{model:max_social_welfare} using a complementarity framework, the approach used in the next subsection.
This derivation allows one to interpret the maximum social welfare model \eqref{model:max_social_welfare} through the lens of complementarity, which may offer interesting connections in its own right.

\subsection{Complementarity for Nash-Cournot competition} \label{sec:complementarity_heuristic}

Given the popularity of complementarity-based methods for convex optimization problems, we also consider using a complementarity-based approach as a heuristic for Nash-Cournot competition. 
Representing all decision variables of player $j$'s pooling problem as $\v{x}_j$, player $j$'s feasible region can be expressed compactly as 
$\mc{X}_j = \mc{X}_j^C \cap \{ \v{x}_j \in \Re^{n_j}: x_{j\ell} \in \Z ~\forall \ell \in \mc{L}_j \}$,
where 
$\mc{X}_j^C = \left\{ \v{x}_j \in \Re^{n_j}: \v{g}_j(\v{x}_j) \leq \v{0}, \v{h}_j(\v{x}_j) = \v{0}, \v{x}_j \geq \v{0} \right\}$
denotes the continuous relaxation of $\mc{X}_j$ and $\mc{L}_j$ denotes player $j$'s set of discrete decisions.
Moreover, the KKT conditions of this continuous relaxation are given by the complementarity system~\eqref{model:generic_nlp_kkt}.

In the Nash-Cournot setting, given a set of linear inverse demand functions, player $j$'s objective function~\eqref{objfnc:nash_cournot_indep_linear} can be represented more generally and concisely as 
\begin{equation} \label{objfnc:nash_cournot_indep_linear_x}
\tilde{p}_j(\v{x}_j,\v{x}_{-j}) = 
\sum_{r \in \mc{R}} \left( \alpha_r - \beta_r \sum_{i \in \mc{J}} x_{ir} \right) x_{jr} - \v{C}_j^{\top} \v{x}_j  
\end{equation}
where the set $\mc{R}$ indexes the output streams where revenue is earned, i.e., $\mc{R} = \Nout$, and $\v{C}_j$ denotes a vector of cost parameters.
This leads to a compact representation of player $j$'s \textit{relaxed} optimization problem, still parameterized by $\v{x}_{-j}$:
\begin{equation} \label{model:relaxed_individual_player_NC}
\max_{\v{x}_j \in \mc{X}^C_j}~ \sum_{r \in \mc{R}} \left( \alpha_r - \beta_r \sum_{i \in \mc{J}} x_{ir} \right) x_{jr} - \v{C}_j^{\top} \v{x}_j 
\end{equation}
Next consider the continuous nonconvex QCQP optimization problem
\begin{subequations} \label{model:relaxed_monolithic_NC}
\begin{alignat}{4}
\max_{\v{x}=(\v{x}_1,\dots,\v{x}_N)}~~ & \tilde{p}^{\textrm{MonoNC}}(\v{x}) = \sum_{j \in \mc{J}} \sum_{r \in \mc{R}} \Bigg[ \alpha_r - \beta_r \left( x_{jr} + \tfrac{1}{2} \sum_{i \in \mc{J}: i \neq j} x_{ir} \right) \Bigg] x_{jr} - \sum_{j \in \mc{J}} \v{C}_j^{\top} \v{x}_j & & \\
\st~~    & \v{x}_j \in \mc{X}^C_j \quad \forall j \in \mc{J} & & 
\end{alignat}
\end{subequations}
Here ``MonoNC'' stands for ``Monolithic Nash-Cournot.''  
Note that the objective function in \eqref{model:relaxed_monolithic_NC} is not the same as simply summing \eqref{model:relaxed_individual_player_NC} over all players. 
For the special case of linear inverse demand functions, Model~\eqref{model:relaxed_monolithic_NC} is important due to the following proposition.
\begin{proposition}
The KKT conditions of Model~\eqref{model:relaxed_monolithic_NC} are identical to the KKT system obtained from aggregating the KKT conditions of Model~\eqref{model:relaxed_individual_player_NC} for all players $j \in \mc{J}$. 
\end{proposition}
\proof 
The KKT conditions of Model~\eqref{model:relaxed_individual_player_NC} are identical to those given by the complementarity system~\eqref{model:generic_nlp_kkt} with $\v{x}_{-j}$ replacing $\vpi$. 
With the exception of the term $\nabla_{\v{x}_j} \tilde{p}_j^{\textrm{MonoNC}}(\v{x})$ replacing the term $\nabla_{\v{x}_j} \tilde{p}_j(\v{x}_j,\v{x}_{-j})$ in \eqref{eq:partial_Lagrangian}, the KKT conditions of Model~\eqref{model:relaxed_monolithic_NC} are the same as those obtained from aggregating \eqref{model:generic_nlp_kkt} for all players.
Finally, since
\begin{equation} \label{eq:partial_derivative_of_p_linear_case}
\frac{\partial \tilde{p}_j(\v{x}_j,\v{x}_{-j})}{\partial x_{jr}} = 
\frac{\partial \tilde{p}_j^{\textrm{MonoNC}}(\v{x})}{\partial x_{jr}} = \alpha_r - \beta_r \left( 2x_{jr} + \sum_{i \in \mc{J}: i \neq j} x_{ir} \right) - C_{jr}
\end{equation}
for $\tilde{p}_j$ as defined in \eqref{objfnc:nash_cournot_indep_linear_x},
the proposition holds.
\qed

When the inverse demand functions are nonlinear, it is difficult, if not impossible, to formulate a single  monolithic model as in \eqref{model:relaxed_monolithic_NC}.  Instead, one must actually write the KKT conditions for all players and then attempt to solve the resulting KKT system as in \eqref{model:generic_nlp_kkt}. In this special case, however, we can write a single optimization to attempt to heuristically find equilibria.  This approach is a heuristic because in a nonconvex setting, the KKT conditions are necessary, but not sufficient. 

When integer decisions are present, we attempt to use monolithic Model~\eqref{model:relaxed_monolithic_NC} by simply re-imposing integrality, i.e., by enforcing $x_{j\ell} \in \Z ~\forall \ell \in \mc{L}_j$. Gabriel et al. \cite{Gabriel2013} employ a different, but closely related heuristic technique when handling discrete decisions by attempting to balance the tradeoff between complementarity and integrality.

\subsection{The fallacy of relying on KKT conditions for nonconvex lower-level problems} \label{sec:fallacy}

It is common in bilevel optimization to replace a lower-level convex problem by its first-order KKT conditions, which are necessary and sufficient. 
Here we provide two simple, concrete examples demonstrating that this approach is unreliable for \textit{nonconvex} lower-level problems.  


\cmt{
\subsubsection{Example with two continuous nonconvex players where complementarity with first-order KKT conditions fails}
\label{sec:kkt_fails}
}
\cmt{
As a variant of the previous example, consider a continuous nonconvex two-player Nash-Cournot game where player $j \in \{1,2\}$ solves  
\begin{equation}
\tilde{p}_j(x_{-j}) = \max \{ -x_j^2 x_{-j} : x_j \in \Re \} = \min \{ x_j^2 x_{-j} : x_j \in \Re \}.
\end{equation}
This game has multiple equilibria given by the set $\{\v{x} \in \Re^2 : x_1 x_2 = 0, x_1 \geq 0, x_2 \geq 0 \}$.
Note that because each player's decision variable is unconstrained, player $j$'s problem is potentially a nonconvex quadratic optimization problem.
Player $j$'s first-order KKT condition is
\begin{equation}
2 x_j x_{-j} = 0,
\end{equation}
revealing that if only first-order KKT conditions are used, as is typically done in complementarity modeling, any solution $(x_1,x_2) \in \{\v{x} \in \Re^2 : x_1 x_2 = 0\}$ is a candidate for an equilibrium. In other words, a complementarity approach using only first-order KKT conditions would yield a superset of candidate equilibria. Note that the set of equilibria could be obtained if one were to include player $j$'s second-order necessary KKT condition (i.e., positive semidefiniteness): $2 x_{-j} \geq 0$. It is not clear if a monolithic formulation exists for this problem.
}

\subsubsection{Example with binary knapsack players}
Consider a multi-player discretely-constrained Nash-Cournot game in which player $j \in \mc{J}$ solves
\cmt{
\begin{equation} \label{model:knapsack_NCgame}
\tilde{p}_j^*(\v{x}_{-j}) = 
\max_{\v{x}_j} \left\{ \tilde{p}_j(\v{x}_{j},\v{x}_{-j}) = \sum_{l \in \mc{L}_j} \left[ \left( \breve{\alpha}_l - \breve{\beta}_l \sum_{i \in \mc{J}} x_{il} \right) x_{jl} - \breve{C}_{jl} x_{jl} \right] : \sum_{l \in \mc{L}_j} a_{jl} x_{jl} \leq b_j, \v{x}_j \in \{0,1\}^{n_j}  \right\}  
\end{equation}
where $\breve{\alpha}_l$ and $\breve{\beta}_l$ represent the parameters of a linear inverse demand function of market $l$, $\mc{L}_j$ denotes the set of markets available to player $j$, $\breve{C}_{jl}$ and $a_{jl}$ denote the cost incurred and resources consumed when player $j$ chooses to participate in market $l$, i.e., and $b_j$ is the total resource amount available to player $j$. The breve notation (e.g., $\breve{\alpha}_l$ ) is meant to avoid confusion with previously defined symbols.
}

This knapsack game is interesting for several reasons.  
\cmt{
First, it has no continuous variables, only binaries, which are the only nonconvex structures in this setting. Second, Dragotto and Scatamacchia \cite{dragotto2021zero} prove that deciding if an instance of this knapsack game has a pure Nash equilibrium -- even with only 2 players -- is $\Sigma_2^p$-complete in the polynomial hierarchy\footnote{Johannes \cite{johannes2011new} provides a nice introduction to the complexity class $\Sigma_2^p$: ``The polynomial hierarchy provides a proper classification scheme for decision problems that appear to be harder than NP-complete. With P and NP at the bottom of the polynomial hierarchy, the next most interesting class is arguably $\Sigma_p^2$.... The complexity class $\Sigma_p^2$ lies one level above the class NP and contains all decision problems that can be solved efficiently by a nondeterministic algorithm that has access to an NP oracle.''}.  This complexity result is a specialization of a more general theorem for integer programming games \cite{carvalho2020computing}. Note that our generalized pooling problem under imperfect competition is an example of a nonconvex nonlinear integer programming game.
}

Since KKT conditions rely on continuous decision variables and differentiable functions, we apply a standard complementarity ``trick'' used commonly in the nonlinear optimization community to handle binary decisions without reverting to branch-and-bound methods.
We treat each binary variable $x_{jl}$ as a continuous variable $x_{jl} \in [0,1]$ subject to the additional complementarity constraint $x_{jl}(1-x_{jl}) = 0$.  In shorthand, we could write $0 \leq x_{jl} \perp (1-x_{jl}) \geq 0$. Applying this substitution gives rise to the nonconvex continuous formulation 
\cmt{
\begin{subequations} \label{model:knapsack_nlp_NCgame}
\begin{alignat}{4}
\max_{\v{x}_j}~~& \sum_{l \in \mc{L}_j} \left[ \left( \breve{\alpha}_l - \breve{\beta}_l \sum_{i \in \mc{J}} x_{il} \right) x_{jl} - \breve{C}_{jl} x_{jl} \right] & & \qquad { \textrm{\underline{Dual vars}}} \\ 
\st~~& \sum_{l \in \mc{L}_j} a_{jl} x_{jl} \leq b_j & & \qquad { \pi_j \geq 0} \\
	& x_{jl}(1-x_{jl}) = 0 & & \qquad { \gamma_{jl} \in \Re} \qquad \forall l \in \mc{L}_j \\	
	& -x_{jl} \geq -1 & & \qquad { \mu_{jl} \geq 0} \qquad \forall l \in \mc{L}_j \\
	 & x_{jl} \geq 0 & & \qquad { \nu_{jl} \geq 0} \qquad \forall l \in \mc{L}_j 
\end{alignat}
\end{subequations}
}
The corresponding Lagrangian function for each player $j$ is
\cmt{
\begin{eqnarray*}
L_j(\v{x}_j,\pi_j,\vgamma_j,\vmu_j,\vnu_j) & = &
\sum_{l \in \mc{L}_j} \left[ \left( \breve{\alpha}_l - \breve{\beta}_l \sum_{i \in \mc{J}} x_{il} \right) x_{jl} - \breve{C}_{jl} x_{jl} \right] 
+ \pi_j \left( b_j - \sum_{l \in \mc{L}_j} a_{jl} x_{jl} \right)
\\ 
& & + \sum_{l \in \mc{L}_j} \gamma_j x_{jl}(1-x_{jl})
+ \sum_{l \in \mc{L}_j} \mu_{jl} (1-x_{jl})
+ \sum_{l \in \mc{L}_j} \nu_{jl} x_{jl}~.
\end{eqnarray*}
}
The necessary, but not sufficient, first-order KKT optimality conditions associated with the nonlinear model~\eqref{model:knapsack_nlp_NCgame} when all players are represented by a monolithic system of equations are
\cmt{
\begin{subequations} \label{model:knapsack_nlp_kkt_NCgame}
\begin{alignat}{4}
\frac{dL_j}{dx_{jl}} = \breve{\alpha}_l - \breve{C}_{jl} - \breve{\beta}_l \sum_{i \neq j} x_{il} - 2 \breve{\beta}_l x_{jl} 
- \pi_j a_j + \gamma_{jl} - 2\gamma_{jl} x_{jl} - \mu_j + \nu_j = 0 & & \qquad \forall j \in \mc{J},l \in \mc{L}_j \\ 
0 \leq b_j - \sum_{l \in \mc{L}_j} a_{jl} x_{jl} \perp \pi_j \geq 0 & & \forall j \in \mc{J} \\
\sum_{l \in \mc{L}_j} \gamma_{jl} x_{jl}(1-x_{jl}) = 0 & & \qquad \forall j \in \mc{J},l \in \mc{L}_j \\	
0 \leq 1-x_{jl} \perp \mu_{jl} \geq 0 & & \qquad \forall j \in \mc{J},l \in \mc{L}_j \\
0 \leq x_{jl} \perp \nu_{jl} \geq 0 & & \qquad \forall j \in \mc{J},l \in \mc{L}_j
\end{alignat}
\end{subequations}
}

The following proposition shows that if one were to choose \textit{any} feasible solution $\hat{\v{x}}_j$ to player $j$'s binary knapsack problem~\eqref{model:knapsack_NCgame} and concatenate all such solutions into a single decision vector $\hat{\v{x}}=(\hat{\v{x}}_1,\dots,\hat{\v{x}}_N) \in \{0,1\}^{\sum_j n_j}$, then $\hat{\v{x}}$ trivially satisfies the KKT conditions~\eqref{model:knapsack_nlp_kkt_NCgame}. More importantly, $\hat{\v{x}}$ need not satisfy the requirement $\tilde{p}_j^*(\hat{\v{x}}_{-j}) \geq \tilde{p}_j(\v{x}_j,\hat{\v{x}}_{-j})$ for all $j \in \mc{J}$ and $\v{x}_j \in \{0,1\}^{n_j}$ satisfying \cmt{$\sum_{l \in \mc{L}_j} a_{jl} \hat{x}_{jl} \leq b_j$}, which, of course, is the definition of a Nash equilibrium.  In short, solving \eqref{model:knapsack_nlp_kkt_NCgame} directly gives no guarantee that an equilibrium has been found. 
\begin{proposition}
Let $\hat{\v{x}}_j$ be any feasible solution to player $j$'s binary knapsack problem~\eqref{model:knapsack_NCgame}.  Let $\hat{\v{x}}=(\hat{\v{x}}_1,\dots,\hat{\v{x}}_N)$. Then there exists $(\hat{\vpi},\hat{\vgamma},\hat{\vmu},\hat{\vnu})$ such that $(\hat{\v{x}},\hat{\vpi},\hat{\vgamma},\hat{\vmu},\hat{\vnu})$ is a feasible solution to the KKT conditions \eqref{model:knapsack_nlp_kkt_NCgame}. 
\end{proposition}
\begin{proof}
Consider any feasible binary solution $\hat{\v{x}}=(\hat{\v{x}}_1,\dots,\hat{\v{x}}_N) \in \{0,1\}^{\sum_j n_j}$ satisfying \cmt{$\sum_{l \in \mc{L}_j} a_{jl} \hat{x}_{jl} \leq b_j$} for all $j \in \mc{J}$.
Let $\hat{\pi}_j = 0$ for all $j \in \mc{J}$.
If $\hat{x}_{jl} = 0$, let $\hat{\mu}_{jl} = 0$, and let $(\hat{\nu}_{jl},\hat{\gamma}_{jl})$ satisfy $\hat{\nu}_{jl} \geq 0$ and \cmt{$\hat{\gamma}_{jl} = -(\breve{\alpha}_l - \breve{C}_{jl} - \breve{\beta}_l \sum_{i \neq j} \hat{x}_{il}) - \hat{\nu}_{jl}$}.  
If $\hat{x}_{jl} = 1$, let $\hat{\nu}_{jl} = 0$, and let $(\hat{\mu}_{jl},\hat{\gamma}_{jl})$ satisfy $\hat{\mu}_{jl} \geq 0$ and \cmt{$\hat{\gamma}_{jl} = \breve{\alpha}_l - \breve{C}_{jl} - \breve{\beta}_l \sum_{i \neq j} \hat{x}_{il} - 2\breve{\beta}_l - \hat{\mu}_{jl}$}.
Then $(\hat{\v{x}},\hat{\vpi},\hat{\vgamma},\hat{\vmu},\hat{\vnu})$ is a KKT solution to \eqref{model:knapsack_nlp_kkt_NCgame}.
\qed
\end{proof}

Gabriel et al. \cite{Gabriel2013} provide an alternative heuristic, i.e., an algorithm with no guarantee of finding an equilibrium, for such problems. 
In contrast, in the following section we describe a method that provably finds an equilibrium if one exists; otherwise it provides a certificate that no equilibrium exists.

\section{Algorithms} \label{sec:algorithms}

\subsection{Decomposition algorithm to minimize disequilibrium}

In this section, we describe and adapt a provably optimal outer approximation method, introduced in \cite{harwood2021equilibrium}, for solving pooling problems under perfect and imperfect competition.
In each iteration, the decomposition algorithm attempts to minimize the disequilibrium between individual player's optima and the ``system" optimum. The algorithm terminates in one of two conditions: (1) No disequilibrium exists, which implies that an equilibrium has been found. (2) Disequilibrium is strictly positive, meaning that no equilibrium exists.
Most importantly, this algorithm is valid even if the individual player problems are nonconvex. 

Consider a set $\Players$ of independent optimization problems, each parameterized by a given vector $\vpi$.
Player $j$'s problem is:
\begin{equation} \label{model:agent_problem_Nash}
p_{j}^*(\vpi) = \max \left\{p_j(\v{x}_j,\vpi): \v{x}_j \in \mc{X}_j \right\}.
\end{equation}
In the spirit of a typical decomposition algorithm, we refer to problem~\eqref{model:agent_problem_Nash} as a \textit{subproblem} or as the $j$th \textit{subproblem}. 
Additionally, consider the following global constraints linking $\vpi$ and individual player decisions $\v{x}_j$ for all $j \in \mc{J}$:
\begin{equation} \label{eq:linking_constraints_Nash}
(\vpi,\x{}) \in \Pi
\end{equation}
To find an equilibrium, consider the following \textit{minimum disequilibrium model} reminiscent of \eqref{blp_MinDiseq}:
\begin{alignat}{3} \label{model:minimum_disequilibrium}
\eta = 
\min_{\vpi,\x{}} 	& \sum_{j \in \mc{J}} \big( p_j^*(\vpi) - p_j(\v{x}_j,\vpi) \big) \notag \\
\st~~ 				
       & (\vpi,\x{}) \in \Pi \\
\notag & \v{x}_j \in \mc{X}_j  \quad \forall j \in \mc{J}
\end{alignat}
Here, the scalar $\eta$ represents the amount of disequilibrium, or the difference between the payoff (profit) in the socially optimal outcome and the individually optimal decision, summed over all players.
Introducing an auxiliary decision variable $w_j$ to denote player $j$'s optimal payoff $p_j^*(\vpi)$ yields an equivalent formulation
\begin{alignat}{3} \label{model:minimum_disequilibrium_lifted}
\eta = 
\min_{\substack{\vpi,\x{},\\\w{}}}~~& \sum_{j \in \mc{J}} \big( w_j - p_j(\v{x}_j,\vpi) \big) \notag \\
\st \qquad 				
	   & (\vpi,\x{}) \in \Pi \\
\notag & \v{x}_j \in \mc{X}_j   \quad \forall j \in \mc{J}  \\
\notag & w_j \geq p_j^*(\vpi) 	\quad \forall j \in \mc{J}.
\end{alignat}
Notice that for any feasible $\vpi$, $w_j \geq p_j^*(\vpi) = p_j(\v{x}_j^*(\vpi),\vpi) \geq p_j(\v{x}_j,\vpi)~\forall \v{x}_j \in \mc{X}_j$, where $\v{x}_j^*(\vpi)$ is an optimal solution for player $j$ given $\vpi$.

Given non-empty finite sets $\mc{X}_j^L \subset \mc{X}_j$ for all $j \in \mc{J}$,
consider the following \textit{relaxed} minimum disequilibrium problem:

\cmt{
\begin{subequations} \label{model:minimum_disequilibrium_relaxed}
\begin{alignat}{4}
\eta^L = 
\min_{\substack{\vpi,\x{},\\\w{}}}~~& \sum_{j \in \mc{J}} \big( w_j - p_j(\v{x}_j,\vpi) \big) \\
\st \qquad
& (\vpi,\x{}) \in \Pi 	 \\
& \v{x}_j \in \mc{X}_j 	 && \quad \forall j \in \mc{J} \\
& w_j \geq p_j(\bar{\v{x}}_j,\vpi) && \quad \forall \bar{\v{x}}_j \in \mc{X}_j^L, j \in \mc{J}. \label{eq:equilibrium_cut}
\end{alignat}
\end{subequations}
}We refer to problem~\eqref{model:minimum_disequilibrium_relaxed} as the \textit{relaxed master problem} (RMP) or \textit{lower bounding problem}.
\cmt{One can think of constraints \eqref{eq:equilibrium_cut} as \textit{equilibrium cuts} as they attempt to ensure that each player's approximate or predicted optimal payoff $w_j$ is at least $p_j^*(\vpi)$, its true optimal payoff given $\vpi$. Without constraints \eqref{eq:equilibrium_cut}, it is possible that the RMP~\eqref{model:minimum_disequilibrium_relaxed} yields an optimum for which $w_j < p_j^*(\vpi)$ indicating that player $j$'s predicted optimal payoff $w_j$ under-approximates $p_j^*(\vpi)$ and hence equilibrium has not been achieved.}

Algorithm~\ref{algo:cut_generating_minimum_disequilibrium} describes an iterative algorithm in which the relaxed master problem~\eqref{model:minimum_disequilibrium_relaxed} is solved in each iteration to obtain a lower bound $\eta^L$ on the minimum disequilibrium (see step~\ref{step:solve_relaxed_master}), after which the individual subproblems are solved (see step~\ref{step:solve_subproblems}) to produce a potentially superior upper bound $\eta^U$ on the minimum disequilibrium. The algorithm terminates once the upper and lower bounds are below a pre-defined optimality tolerance $\epsilon$. 

\begin{algorithm}
\caption{Cutting plane algorithm to minimize disequilibrium}
\label{algo:cut_generating_minimum_disequilibrium}
\begin{algorithmic}[1]
\REQUIRE Optimality tolerance $\epsilon > 0$; non-empty finite sets $\mc{X}_j^L$ of feasible solutions for all $j \in \mc{J}$
\STATE Set $\eta^L = -\infty$, $\eta^U = +\infty$, $\vpi^* = \v{0}$, $\v{x}^* = \v{0}$;
\WHILE{$(\eta^U - \max\{\eta^L,0\} > \epsilon)$}
	\STATE Solve the relaxed master problem~\eqref{model:minimum_disequilibrium_relaxed} to obtain ($\hat{\vpi},\hatx{},\hatw{})$ and lower bound $\eta^L$ \label{step:solve_relaxed_master}
	\FOR{each player $j \in \mc{J}$}
		\STATE Solve model~\eqref{model:agent_problem_Nash} as a function of $\hat{\vpi}$ for optimal solution $\bar{\v{x}}_j$ and objective value $p_j^*(\hat{\vpi})$ \label{step:solve_subproblems}
	\ENDFOR
	\IF{$\hat{w}_j \geq p_j^*(\hat{\vpi})~\forall j \in \mc{J}$} \label{step:master_optimality_check}
		\STATE $\vpi^* = \hat{\vpi}$, $\v{x}^* = \hat{\v{x}}$
		\RETURN $(\vpi^*,\v{x}^*_1,\dots,\v{x}^*_{|\mc{J}|})$
	\ELSE
		\STATE $\mc{X}_j^L = \mc{X}_j^L \cup \{\bar{\v{x}}_j\}$ 
	\ENDIF
	\STATE \textbf{if} $\sum_{j \in \mc{J}} ( p_j^*(\hat{\vpi}) - p_j(\hat{\v{x}}_j,\hat{\vpi}) ) < \eta^U$, \textbf{then} $\eta^U = \sum_{j \in \mc{J}} ( p_j^*(\hat{\vpi}) - p_j(\hat{\v{x}}_j,\hat{\vpi}) ); \vpi^* = \hat{\vpi}$, $\v{x}^* = \hat{\v{x}}$; 
\ENDWHILE
\RETURN $(\vpi^*,\v{x}^*_1,\dots,\v{x}^*_{|\mc{J}|})$
\end{algorithmic}
\end{algorithm}
Algorithm~\ref{algo:cut_generating_minimum_disequilibrium} can terminate in one of two ways.
Either:
\begin{enumerate}\itemsep0pt \parskip0pt
\item
An optimal solution to the relaxed master problem~\eqref{model:minimum_disequilibrium_relaxed} at a certain iteration is feasible in Problem~\eqref{model:minimum_disequilibrium_lifted}, indicated by $\hat{w}_j \geq p_j^*(\hat{\vpi})$ for each $j \in \mc{J}$
(note that this relies on the global minimizer of \eqref{model:minimum_disequilibrium_relaxed} being found, otherwise some $w_j$ might be ``too large'').
Then the solution to the relaxed master problem~\eqref{model:minimum_disequilibrium_relaxed} is optimal for \eqref{model:minimum_disequilibrium_lifted}, since \eqref{model:minimum_disequilibrium_relaxed} is a relaxation.
Thus $\eta^L = \eta$ and $(\vpi^*,\v{x}^*)$ is a solution.
\item
An $\epsilon$-optimal solution $(\vpi^*,\v{x}^*)$ of Problem~\eqref{model:minimum_disequilibrium} is found.
Note that
$\sum_{j \in \mc{J}} (p_j^*(\vpi) - p_j(\v{x}_j,\vpi))$
is an upper bound for $\eta$;
we have merely evaluated the objective at a feasible point of Problem~\eqref{model:minimum_disequilibrium}.
The upper bound $\eta^U$ tracks the best of these upper bounds.
\end{enumerate}

If one only cares about the existence of an equilibrium solution, the algorithm can terminate early if it happens that $\eta^L > 0$.
Since $\eta^L$ is a lower bound on $\eta$ by construction, an equilibrium solution does not exist if $\eta^L > 0$ \cite{harwood2021equilibrium}.  
It is also worth noting, particularly for large problems, one can pre-populate the set $\mc{X}^L$ with candidate solutions with the hope of reducing the number of iterations, i.e., by allowing the relaxed master problem to find better solutions more quickly.  
Finally, the algorithm is guaranteed to converge to an $\epsilon$-optimal solution in finite iterations when the functions $p_j$ are continuous, and the sets $\Pi$ and $\mc{X}_j$ for all $j \in \mc{J}$ are compact and nonempty \cite{harwood2021equilibrium}.

\subsubsection{Decomposition applied to competitive pooling}

When specialized to our competitive pooling problem setting,
let $\v{x}_j = (\v{c}_j,\v{f}_j,\v{u}_j)$ for all $j \in \mc{J}$ (pooling players, not suppliers).
In the Nash-Cournot setting, $\Pi$ takes the form
\begin{equation}
\Pi_{\textrm{NC}} = \left\{ (\vpi,\x{}) \in \Re^{d} \times \Re^{\sum_j n_j}: \vpi_j = \v{x}_j~\forall j \in \mc{J} \right\}
\end{equation}
and the function $p_j(\bar{\v{x}}_j,\vpi)$ becomes
\begin{equation} \label{objfnc:nash_cournot_indep_linear_with_pi}
\sum_{n \in \Nout_j} \left[ \alpha_n - \beta_n \left( \bar{f}_{jn}^{\OUT} + \sum_{i \in \mc{J}: i \neq j} f_{in}^{\OUT} \right) \right] \bar{f}_{jn}^{\OUT} - \sum_{n \in \Nin_j} \left( C_{jn}^{\textrm{var}} \bar{f}_{jn}^{\IN} + C_{jn}^{\textrm{fixed}} \bar{u}_{jn} \right). 
\end{equation}
In the price taker (PT) (a.k.a. perfect competition) setting, $\Pi$ takes the form
\begin{equation}
\Pi_{\textrm{PT}} = \left\{ (\vpi,\v{q},\x{}) \in \Re^{d} \times \Re^{|\Nout|} \times \Re^{\sum_j n_j}: q_n = \sum_{j \in \mc{J}} f_{jn}^{\OUT} \right\}
\end{equation}
and the function $p_j(\bar{\v{x}}_j,\vpi)$ becomes
\begin{equation} \label{objfnc:price_taker_indep_linear_with_pi}
\sum_{n \in \Nout_j} \pi_n \bar{f}_{jn}^{\OUT} - \sum_{n \in \Nin_j} \left( C_{jn}^{\textrm{var}} \bar{f}_{jn}^{\IN} + C_{jn}^{\textrm{fixed}} \bar{u}_{jn} \right). 
\end{equation}
In this setting, we can embed the consumer into the master problem by taking its KKT conditions
\begin{equation} \label{eq:consumer_stationary_constraints}
\pi_n = \alpha_n - \beta_n \sum_{j \in \mc{J}} f_{jn}^{\OUT} \qquad \forall n \in \Nout
\end{equation} 
and including them in the global constraint set $\Pi$.
If an equilibrium does not exist and we embed constraints \eqref{eq:consumer_stationary_constraints} in $\Pi$, then no disequilibrium will be allocated to the consumer; it will be apportioned only to the suppliers.  However, if an equilibrium does exist, then this approach will return an equilibrium.

\cmtt{
\subsubsection{Accelerating Decomposition Algorithm~\ref{algo:cut_generating_minimum_disequilibrium}}
}
\cmtt{
Decomposition Algorithm~\ref{algo:cut_generating_minimum_disequilibrium} is inspired by semi-infinite programming and Benders decomposition.  In each major iteration of Algorithm~\ref{algo:cut_generating_minimum_disequilibrium}, the RMP~\eqref{model:minimum_disequilibrium_relaxed} generates a candidate solution (equilibrium), which then undergoes an optimality check in Step~\ref{step:master_optimality_check}.  If this check fails, cuts generated by the subproblems are appended to the RMP and the process repeats.  Since the RMP can be challenging to solve to provable optimality for nonconvex problems, it may be time-consuming to wait for the RMP to generate candidate equilibria.  
}
\cmtt{
There are at least two options for diminishing this reliance on the RMP for candidate solutions.
\begin{enumerate}
\item ``Warmstart Phase'': First, if one has access to an effective heuristic (which is commonly the case as discussed below), then one can perform a non-trivial ``warmstart phase'' in an attempt to pre-populate the set $\mc{X}_j^L$ of initial solutions with a number of near equilibria.  This warmstart phase typically reduces the number of major iterations (and, consequently, overall solution time) required to hone in on good solutions.  Of course, one could even stumble upon an equilibrium while performing this step. 
\item ``Perturbation Phase'': In each major iteration, the RMP~\eqref{model:minimum_disequilibrium_relaxed} generates a candidate solution that is guaranteed to be feasible for each player, but that may not satisfy the Nash equilibrium conditions $p_j^*(\vpi) \geq p_j(\v{x}_j,\vpi)$ for all $j \in \mc{J}$. In other words, the RMP is guided by the equilibrium cuts~\eqref{eq:equilibrium_cut} that have been generated, but because it cannot ``see'' all infinite cuts, it proposes candidate solutions that are likely not equilibria, especially in early iterations.  Rather than rely on the RMP to generate near equilibria, one can ``explore'' around the RMP's candidate solution $\hat{\vpi}$ by considering one or more perturbed solutions $\tilde{\vpi} \simeq \hat{\vpi}$.  This can be accomplished by asking a solver to generate a solution pool when solving the RMP; by exploiting problem structure to tweak the current solution; or by some other means.  While these perturbed solutions (or candidate neighbors) may not yield an equilibrium, they can generate useful cuts as the solutions can be appended to the set $\mc{X}_j^L$.   
\end{enumerate}
We explore the first option, but not the second, in this work.
}

\subsection{Complementarity heuristics for perfect and Nash-Cournot competition} \label{sec:heuristics}

For ease of reference, we restate the heuristics used in this work (and described in Sections \ref{sec:perfect_competition} and \ref{sec:complementarity_heuristic}) to compare against our exact decomposition algorithm.  
\subsubsection{Perfect competition}
We solve the maximum social welfare model~\eqref{model:max_social_welfare}. When an equilibrium exists, we obtain prices by setting $\pi_n = \alpha_n - \beta_n \sum_{j \in \mc{J}} f_{jn}^{\OUT}$ for all $n \in \Nout$. Under these assumptions, solving the maximum social welfare model~\eqref{model:max_social_welfare} can be considered an exact method as it is guaranteed to find an equilibrium.  When no equilibrium exists, solving the maximum social welfare model~\eqref{model:max_social_welfare} alone may not detect the non-existence of an equilibrium, which is why we will refer to ``solving the maximum social welfare model~\eqref{model:max_social_welfare}'' as a heuristic.

\subsubsection{Nash-Cournot competition}
When no integer decisions are present, we solve \cmtt{Monolithic Complementarity} Model~\eqref{model:relaxed_monolithic_NC}. As a reminder, the KKT conditions to this continuous optimization problem are necessary, but not sufficient, due to the presence of nonconvex bilinear constraints.  When integer decisions are present, we again attempt to solve Model~\eqref{model:relaxed_monolithic_NC} after re-imposing integrality, i.e., by enforcing $x_{j\ell} \in \Z ~\forall \ell \in \mc{L}_j$. It is for these reasons that we refer to the complementarity approach as a heuristic.

\subsection{Other methods}
There are a number of alternative approaches for heuristically solving complex equilibrium problems, some of which we describe below.
Since we do not formally compare against these methods, we prefer to organize them into one subsection for ease of reference.
In general, the methods listed below are fast and useful when attempting to obtain approximate market prices and player decisions. Their main drawback is that their convergence properties are either poorly understood or non-existent even when all players are convex.  Solution time, however, is a critical factor and we readily admit that the relaxed master problem \eqref{model:minimum_disequilibrium_relaxed} becomes increasingly time consuming as the number of players grows and when individual player constraints are numerous and nonconvex.  
One could potentially decompose the RMP via a penalization or ADMM scheme, but this is beyond the scope of this work. On the other hand, a potential benefit of attempting to minimize disequilibrium via decomposition is that in the course of solving the RMP~\eqref{model:minimum_disequilibrium_relaxed}, if the lower bound $\eta_L$ on disequilibrium ever exceeds zero, then we have a proof that no equilibrium exists.  No such guarantee is provided by a heuristic.

\subsubsection{Mesh grid approximation} \label{sec:mesh_grid}

For problems in which $\vpi$ is low-dimensional, one could attempt to decompose the minimum disequilibrium model \eqref{model:minimum_disequilibrium} into a first- and second-stage problem as follows: 
\begin{equation} \label{model:eta_of_pi}
\eta = 
\min_{\vpi \in \Re^d} \eta(\vpi),
\end{equation}
where
\begin{alignat}{3} \label{model:minimum_disequilibrium_restricted_given_pi}
\eta(\vpi) =
\min_{\x{}} 	& \sum_{j \in \mc{J}} \big( p_j^*(\vpi) - p_j(\v{x}_j,\vpi) \big) \notag \\
\st~~ 				
       & (\vpi,\x{}) \in \Pi \\
\notag & \v{x}_j \in \mc{X}_j  \quad \forall j \in \mc{J}
\end{alignat}
For example, when $\vpi$ is 1- or 2-dimensional, then it may be possible to construct a mesh grid or discretized grid of $\vpi$ vectors and solve for the disequilibrium $\eta(\vpi)$ to approximate the minimum disequilibrium. 
Specifically, once $\vpi$ is fixed, one can solve player $j$'s subproblem for $p_j^*(\vpi)$, which then becomes a constant in the objective function of \eqref{model:minimum_disequilibrium_restricted_given_pi}. The $\eta(\vpi)$ problem \eqref{model:eta_of_pi} essentially decomposes by player $j$ except for a single constraint set $\Pi$ that links them together. 

\subsubsection{Black box search}

For any $\v{x}_j \in \mc{X}_j$ and $\vpi$, we may define player $j$'s disequilibrium as 
\begin{equation} \label{eq:disequilibrium_def}
\delta_j(\v{x}_j,\vpi) = p_{j}^*(\vpi) - p_j(\v{x}_j,\vpi) \quad \forall j \in \mc{J}.
\end{equation}
Note that, by definition, $\delta_i(\v{x}_i,\vpi) \geq 0$.
Then, a point $(\vpi,\x{}) \in \Pi$ is a Nash equilibrium if 
\begin{equation}
\delta_j(\v{x}_j,\vpi) = 0 \quad \forall j \in \mc{J}.
\end{equation}
Note that our intuitive notion/definition of disequilibrium is related (equivalent) to the Nikaido-Isoda function commonly used in equilibrium modeling \cite{harwood2021equilibrium}.
Alternatively, for any $\epsilon > 0$, a point $(\vpi,\x{}) \in \Pi$ is an $\epsilon$-Nash equilibrium if 
\begin{equation}
\max\{ \delta_j(\v{x}_j,\vpi) : j \in \mc{J} \} < \epsilon~.
\end{equation} 
One could just as well perform a black box search over $(\vpi,\x{}) \in \Pi$ and $\v{x}_j \in \mc{X}_j$ to find an $\epsilon$-Nash equilibrium.  This generic view of the problem allows one to consider genetic algorithm, gradient descent, simulated annealing, or any other black box solver of interest. In a nonconvex setting, convergence guarantees to a global optimum are rarely available for black box solvers.

\subsubsection{Practitioner's method} \label{sec:practitioners_method}

Facchinei and Kanzow \cite[Section 5.1]{Facchinei2007} describe a popular ``practitioner's method'' for generalized Nash equilibrium problems, i.e., problems in which a player's feasible region and objective function are dependent on other players' decisions.  Due to its popularity and ease of implementation,  we describe the method in pseudocode in Algorithm~\ref{algo:practitioners_method_Jacobi_type} for the Nash-Cournot setting. It begins with an initial solution $\v{x}^0 = (\v{x}_1^0,\dots,\v{x}_N^0)$ where the superscript indexes the iteration. At each iteration $k$, one solves player $j$'s problem as a function of $\v{x}_{-j}^k$ and retains the resulting maximizer as $\v{x}_j^{k+1}$ for the next iteration. A typical termination criterion is that no player's objective function improves from one iteration to the next, i.e., $\tilde{p}_j(\v{x}_j^{k},\v{x}_{-j}^{k}) \geq \tilde{p}_j(\v{x}_j^{k+1},\v{x}_{-j}^{k})$ for all $j \in \mc{J}$. 
\begin{algorithm}
\caption{Nonlinear Jacobi-type Method}
\label{algo:practitioners_method_Jacobi_type}
\begin{algorithmic}[1]
\STATE choose a starting point $\v{x}^0 = (\v{x}_1^0,\dots,\v{x}_1^N)$; set iteration counter $k=0$
\WHILE{termination criteria are not satisfied}
\FOR{each player $j \in \mc{J}$}
	\STATE Solve player $j$'s subproblem for $\v{x}_j^{k+1} \in \arg\max\{ \tilde{p}_j(\v{x}_j,\v{x}_{-j}^{k}) : \v{x}_j \in \mc{X}_j \}$ \label{step:solve_subproblems_practitioners_method_Jacobi}
\ENDFOR
\STATE $k = k+1$;
\ENDWHILE
\end{algorithmic}
\end{algorithm}

Algorithm~\ref{algo:practitioners_method_Jacobi_type} is considered a Jacobi-type method because Step~\ref{step:solve_subproblems_practitioners_method_Jacobi} uses the vector $\v{x}_{-j}^{k}$ to determine the newest maximizer.  This step involves the solution of $N$ independent subproblems and can be executed in parallel.  Alternatively, one could apply a Gauss-Seidel-type update in which player $j$'s subproblem depends on $(\v{x}_1^{k+1},\dots,\v{x}_{j-1}^{k+1},\v{x}_{j+1}^{k},\dots,\v{x}_{N}^{k})$ reflecting the latest information available for players $1$ through $j-1$. 

While there are multiple well-documented challenges associated with these heuristics, arguably the most salient difficulty, as Facchinei and Kanzow attest, is ``the convergence properties of [these methods] are not well-understood. ... [A]t present, they can be considered, at most, good and simple heuristics'' \cite[p.196]{Facchinei2007}. \cmtt{As such, we experiment with two very simple variants to generate warmstart solutions (candidate equilibria) when considering our larger-scale instances in Section~\ref{sec:algorithm_scaleup}.
} 

\cmtt{
\begin{algorithm}
\caption{Warmstart Method: Price-taker setting}
\label{algo:warmstart_method_Price_Taker}
\begin{algorithmic}[1]
\STATE Set iteration counter $k=0$; Set $(\v{x}_1^{0},\v{x}_2^{0})=(\v{0},\v{0})$
\STATE Generate random prices: $\pi_n=\alpha_n \texttt{Uniform}(0.05,0.95)$ for $n \in \Nout$, where $\texttt{Uniform}(a,b)$ is a uniform random variable from $a$ to $b$.
\STATE \textbf{for} each player $j \in \mc{J}$, solve player $j$'s subproblem $\max\{ p_j(\v{x}_j,\vpi) : \v{x}_j \in \mc{X}_j \}$ for $\v{x}_j^{1}$ and $p_j^U(\vpi)$
\RETURN $(\v{x}_1^{1},\v{x}_2^{1})$
\end{algorithmic}
\end{algorithm} 
}

\cmtt{Algorithm~\ref{algo:warmstart_method_Price_Taker} describes a ``single-iteration'' heuristic for the price-taker setting in which a vector of random prices is generated and then each player optimizes with respect to these prices; no additional solves are performed because our main goal is to investigate our minimum disequilibrium algorithm~\ref{algo:cut_generating_minimum_disequilibrium}.  Note that it is extremely rare for this approach to generate player solutions $f_{jn}^{\OUT}$ satisfying $\pi_n = \alpha_n-\beta_n \sum_{j \in \mc{J}} f_{jn}^{\OUT}$ for $n \in \Nout$. This is important because it means that, while it is easy to determine if an equilibrium has \textit{not} been achieved, additional actions are needed to produce and/or confirm that an equilibrium has been found.
}

\cmtt{
\begin{algorithm}
\caption{Gauss-Seidel-type Warmstart Method: Nash-Cournot setting}
\label{algo:warmstart_method_Nash_Cournot_Gauss_Seidel_type}
\begin{algorithmic}[1]
\STATE Set iteration counter $k=0$; Set $(\v{x}_1^{0},\v{x}_2^{0})=(\v{0},\v{0})$
\STATE Generate random prices: $\pi_n=\alpha_n \texttt{Uniform}(0.05,0.95)$ for $n \in \Nout$ 
\STATE Fix player 2's output: $f_{2n}^{\OUT} = (\alpha_n - \pi_n)/\beta_n$ for $n \in \Nout$; Update these components in $\v{x}_{2}^{0}$ \label{step:fix_player2s_output}
\STATE Solve player $1$'s subproblem $\max\{ \tilde{p}_1(\v{x}_1,\v{x}_{2}^{0}) : \v{x}_1 \in \mc{X}_1 \}$ for $\v{x}_1^{1}$ \label{step:player1_first_solve}
\STATE Solve player $2$'s subproblem $\max\{ \tilde{p}_2(\v{x}_1^1,\v{x}_{2}) : \v{x}_2 \in \mc{X}_2 \}$ for $\v{x}_2^{1}$ \label{step:player2_first_solve}
\STATE Solve player $1$'s subproblem $\max\{ \tilde{p}_1(\v{x}_1,\v{x}_{2}^{1}) : \v{x}_1 \in \mc{X}_1 \}$ for $\v{x}_1^{2}$ and $\tilde{p}_1^U(\v{x}_{2}^{1})$ \label{step:player1_verification_solve}
\STATE Solve player $2$'s subproblem $\max\{ \tilde{p}_2(\v{x}_1^2,\v{x}_{2}) : \v{x}_2 \in \mc{X}_2 \}$ for $\tilde{p}_2^U(\v{x}_{1}^{2})$ \label{step:player2_verification_solve}
\RETURN $(\v{x}_1^{2},\v{x}_2^{1})$,$(\tilde{p}_1^U(\v{x}_{2}^{1}),\tilde{p}_2^U(\v{x}_{1}^{2}))$
\end{algorithmic}
\end{algorithm} 
}

\cmtt{Algorithm~\ref{algo:warmstart_method_Nash_Cournot_Gauss_Seidel_type} is a particular implementation of the Practitioner's heuristic for the Nash-Cournot setting. After generating random prices at each output node $n \in \Nout$, Step~\ref{step:fix_player2s_output} sets player 2's output $f_{2n}^{\OUT}$ such that $\pi_n = \alpha_n - \beta_n f_{2n}^{\OUT}$ for all $n \in \Nout$. Consequently, the initial solution $(\v{x}_1^{0},\v{x}_2^{0})$ may not be feasible. Steps~\ref{step:player1_first_solve} and \ref{step:player2_first_solve} therefore find feasible solutions $\v{x}_j^1 \in \mc{X}_j$ for each player $j \in \mc{J}$ (in a Gauss-Seidel manner), while Steps~\ref{step:player1_verification_solve} and \ref{step:player2_verification_solve} are needed to compute the upper/dual bounds $\tilde{p}_j^U(\cdot)$.
Why do we return $(\v{x}_1^{2},\v{x}_2^{1})$ and not $(\v{x}_1^{2},\v{x}_2^{2})$?  Because we computed $\tilde{p}_1^U(\v{x}_{2}^{1})$, not $\tilde{p}_1^U(\v{x}_{2}^{2})$, in Step~\ref{step:player1_verification_solve} and we use $\tilde{p}_1^U(\v{x}_{2}^{1})$ in our disequilibrium calculation.  See equation~\eqref{eq:relative_disequilibrium_gap_definition} and the discussion thereafter.     
}

\subsection{Algorithm summary}

\cmtt{
We end this section with a brief qualitative comparison of the methods that we have discussed and which we compare below.  Table~\ref{tbl:Qualitative_algorithm_comparison} attempts to capture fundamental differences in convergence guarantees and relative speed for \textit{nonconvex} games; convex games are irrelevant in this work. We distinguish between two types of convergence: convergence of the algorithm itself to a solution, not necessarily an equilibrium (``Algorithm'' column) and  convergence to an equilibrium or to a certificate that no equilibrium exists (``Equilibrium or proof of non-existence'' column).  The final two columns qualify the speed to arrive at a candidate equilibrium relative to the other methods and where the computational difficulty lies.
}

\cmtt{
The first three rows are associated with heuristics that are relatively fast, but have no guarantee of finding an equilibrium or proving that no equilibrium exists.  Worse, the practitioner's method may not even converge to a solution.  It is worth re-iterating that if one were to consider nonlinear inverse demand functions, then it may be impossible to formulate a monolithic complementarity optimization model, in which case one would have to explicitly write and solve a nonlinear KKT system as discussed in the Introduction. This KKT system is, in effect, a feasibility problem and thus there is no need to distinguish between local and global optimality. Meanwhile, our ``pure'' (i.e., sans warmstarts) minimum disequilibrium algorithm provides convergence guarantees, but at the expense of slower solution times.  This can be remedied by integrating a heuristic at the outset to find high-quality candidate equilibria.  It is also possible to integrate heuristics elsewhere within Algorithm~\ref{algo:cut_generating_minimum_disequilibrium}, although we did not pursue this extension in this work.
}

\begin{table}[h]									
\centering	
{\color{black}									
\begin{tabular}{p{5.5cm}cp{2.1cm}p{2.4cm}p{3cm}}									
\toprule									
Method	&	\multicolumn{2}{c}{Convergence guarantees}			&	Relative speed	&	Speed depends on difficulty of	\\
\cmidrule(lr){2-3}
	&	Algorithm	&	Equilibrium or proof of non-existence 	&		&		\\
\midrule									
Practitioner's method (Algorithm~\ref{algo:practitioners_method_Jacobi_type})	&	No	&	No	&	Fast	&	Subproblems~\eqref{model:agent_problem_Nash} 	\\
Complementarity -- Global	&	Yes	&	No	&	Moderately fast	&	Model~\eqref{model:max_social_welfare_projected} or \eqref{model:relaxed_monolithic_NC}	\\
Complementarity -- Local	&	Yes	&	No	&	Fast	&	Model~\eqref{model:max_social_welfare_projected} or \eqref{model:relaxed_monolithic_NC}	\\
MD (Algorithm~\ref{algo:cut_generating_minimum_disequilibrium})	&	Yes	&	Yes	&	Slow	&	RMP~\eqref{model:minimum_disequilibrium_relaxed} and subproblems~\eqref{model:agent_problem_Nash}	\\
MD + Warmstarts  (Algorithm~\ref{algo:cut_generating_minimum_disequilibrium} + Algorithm~\ref{algo:warmstart_method_Price_Taker} or \ref{algo:warmstart_method_Nash_Cournot_Gauss_Seidel_type})	&	Yes	&	Yes	&	Fast to moderately fast	&	RMP~\eqref{model:minimum_disequilibrium_relaxed} and subproblems~\eqref{model:agent_problem_Nash}	\\
\bottomrule									
\end{tabular}									
\caption{Qualitative algorithm comparison for \textit{nonconvex} problems. ``Complementarity -- Global'' and ``-- Local'' = Solving Model~\eqref{model:max_social_welfare_projected} or \eqref{model:relaxed_monolithic_NC} to global and local optimality, respectively. ``MD'' = our pure Minimum Disequilibrium algorithm with no warmstart solutions ($\mc{X}_j^L = \emptyset$ for all $j \in \mc{J}$).}									
\label{tbl:Qualitative_algorithm_comparison}	
}								
\end{table}									

\section{Numerical results} \label{sec:results}

In this section, 
we present empirical analysis to highlight the reason why one might consider modeling competitive behavior in pooling-related problems or more general mixed-integer nonlinear optimization problems. 
We also show that our cutting plane algorithm for minimizing disequilibrium is capable of solving challenging non-cooperative pooling problems. 
All models and algorithms were coded in AIMMS version \cmtt{4.86.7.5} and solved serially with \cmtt{Gurobi 9.5. All instances (i.e., of the RMP, subproblems, and complementarity models) were declared optimal using a relative optimality gap of 1e-4}.

We begin our experiments with variants of the (in)famous Haverly pooling problem \cite{haverly1978studies} shown in Figure~\ref{fig:haverly_p}. In the traditional Haverly pooling problem, there is a single producer whose goal is to maximize profit by optimizing crude purchases and flows through a network. We extend this simple example to involve multiple producers (players) who compete to sell finished goods (output streams). There is a market for each output stream. 
Each player operates a processing network in which raw materials are purchased and processed (sometimes blended) into final goods that are then sold to maximize profit.
Each player purchases raw materials that enter the network at a set $\mc{N}^{\IN}$ of input nodes; exactly one type of raw material can be purchased at each input node and that type is known at the outset.   
Raw materials are then sent to an intermediate pool $n \in \mc{N}^{\textrm{pool}}$, where two or more raw materials are blended, or directly to an output node $n \in \mc{N}^{\OUT}$.
There is no flow between pools, from a pool to an input node, or from an output node to a pool or input node.
Each raw material has a concentration $C_{jn}^{\IN}$ when entering the network.
The concentration of an output stream at output node $n \in \mc{N}^{\OUT}$ must not exceed $C_{n}^{\max}$.

In our experiments, we assume that raw materials are modeled as semi-continuous variables with a variable and fixed cost component.  That is, if a raw material is purchased at input node $n \in \mc{N}^{\IN}$, then a one-time (lump sum) fixed cost is incurred (like a transaction cost) and a variable cost per unit is also incurred.  
\cmt{Note that, when no binary decision variables are present, Baltean-Lugojan and Misener \cite{baltean2018piecewise} show that the standard pooling problem with a ``single quality standard'' (and under several other conditions which hold for the original Haverly problem) can be solved in strongly-polynomial time. Hence, the player subproblems \eqref{model:agent_problem_Nash} to our standard (continuous) Haverly instances, but not our general (mixed-integer) instances, are solvable in strongly-polynomial time, 
while no such complexity result exists for finding pure Nash equilibria.}

\begin{figure}[h]  
\centering
\includegraphics[width=10cm]{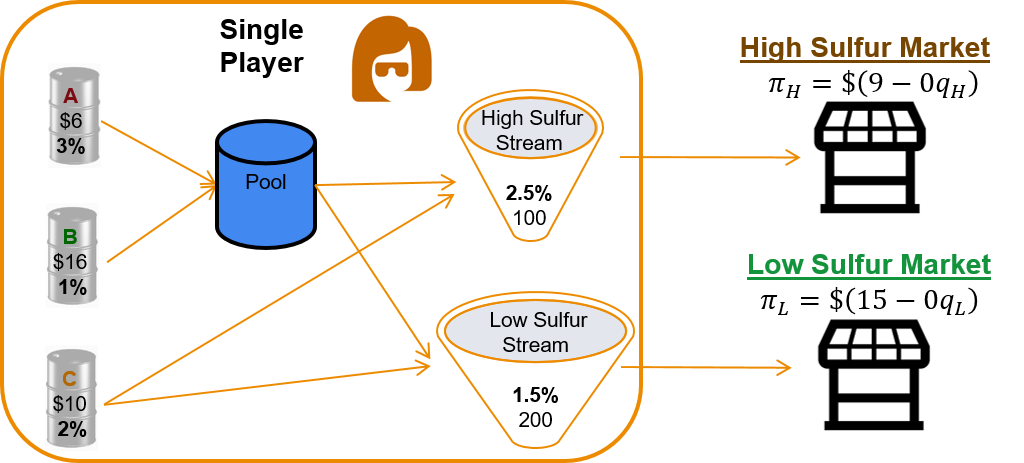}
\caption{Original Haverly pooling problem instance \cite{haverly1978studies}. Each crude has a unit cost and a sulfur concentration. After blending crudes, at most 100 and 200 units of a high and low sulfur stream can be produced with a maximum sulfur concentration of 2.5\% and 1.5\%, respectively. The high and low streams are then sold at a price of $\pi_H=\$9$ and $\pi_L=\$15$ per unit, respectively, regardless of the quantity ($q_H$ and $q_L$) produced. In economic terms, the stream prices are perfectly elastic. The globally optimal solution is to purchase 100 units of crudes B and C and produce 200 units of the low sulfur stream for a profit of \$400.}
\label{fig:haverly_p}
\end{figure}

\subsection{Illustrative example: Comparing perfect and imperfect competition}

To demonstrate why modeling competitive behavior is important, we first consider two Haverly-esque instances involving two non-cooperative players.  As shown in Figure~\ref{fig:Competitive_Haverly_pooling_2player_instance}, player 1's parameters are identical to those in the original Haverly instance. Meanwhile, player 2's variable cost ($C_{1n}^{\textrm{var}}$ for $n \in \Nin$) for crude A, B, and C is \$3, \$18, and \$11, respectively.  All other player 2 parameters match those of player 1.  Both players supply two independent markets each governed by a linear inverse demand function. The high and low sulfur market's parameters are $(\alpha_H,\beta_H)=(13,0.02)$ and $(\alpha_L,\beta_L)=(23,0.04)$, respectively. The low sulfur market's larger (negative) slope implies that demand for this good is more inelastic than that of the high sulfur market. As a point of reference, note that $\pi_H = \alpha_H - \beta_H 200 = 9$ and $\pi_L = \alpha_L - \beta_L 200 = 15$, i.e., if the total quantity supplied to each market is $q_H=q_L=200$, the market prices match those of the original Haverly instance.  

\begin{figure}[h] 
\centering
\includegraphics[width=10cm]{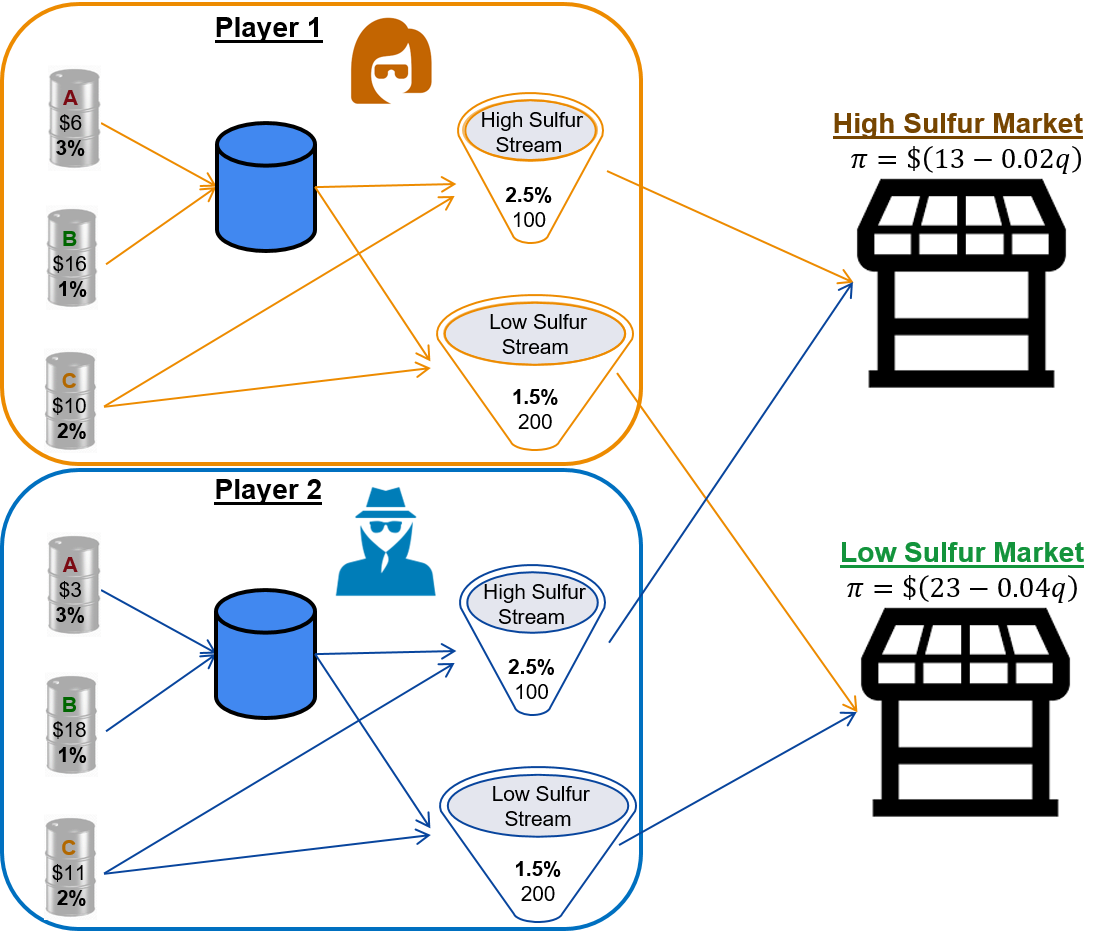}
\caption{Two-player Haverly-esque instance in which non-cooperative players vie to maximize their individuals profits by selling into two markets represented with independent linear inverse demand functions.}
\label{fig:Competitive_Haverly_pooling_2player_instance}
\end{figure}

In our first instance, there are no fixed costs (hence, no binary decisions) for either player, giving rise to a \textit{standard} pooling problem instance.  In the second instance, we introduce a fixed cost of $C_{jB}^{\textrm{fixed}}=200$ for $j=1,2$ if any amount of crude B, the lowest sulfur crude, is purchased. Hence, each player solves a \textit{generalized} pooling problem. For each instance, we consider two competitive settings: perfect competition and imperfect competition. In the former, both players act as price takers unable to anticipate the impact of their decisions on market prices.  In the latter, both players act strategically as Nash-Cournot players capable of anticipating the impact of their decisions on market prices.
 
\textbf{Standard pooling problem instance (Purely continuous setting)}. Figure~\ref{fig:Competitive_Haverly_pooling_2player_solutions_6_16_10_v_3_18_11_0FixedCost} shows an equilibrium solution for each competitive setting. Comparing the consequences of the two assumptions, we see that prices increase when imperfect competition is assumed. Namely, the price of the high sulfur stream increases from \$10 to \$10.5 per unit, while the price of the low sulfur stream increases from \$15 to \$18 per unit. Correspondingly, the total quantity of each decreases. Under the parameters chosen, player 2 makes identical decisions in both settings, whereas player 1 withholds production allowing her profits to increase from \$400 to \$637.5.  This withholding simultaneously allows player 2 to increase his profits from \$325 to \$375.

\textbf{Generalized pooling problem instance (Mixed-integer setting)}. Figure~\ref{fig:Competitive_Haverly_pooling_2player_solutions_6_16_10_v_3_18_11_1FixedCost} shows an equilibrium solution for each competitive setting after a fixed cost has been introduced.  Player 1's decisions are identical to those made in the corresponding instances with no fixed cost.  On the other hand, although player 2 continues to supply 100 units of the high sulfur stream, the presence of a fixed cost for crude B leads him to alter his crude selection. Instead of choosing to blend crudes A and B as shown in Figure~\ref{fig:Competitive_Haverly_pooling_2player_solutions_6_16_10_v_3_18_11_0FixedCost}, it is now more economical to blend crudes A and C.

\begin{figure}[h!] 
\centering
\includegraphics[width=13cm]{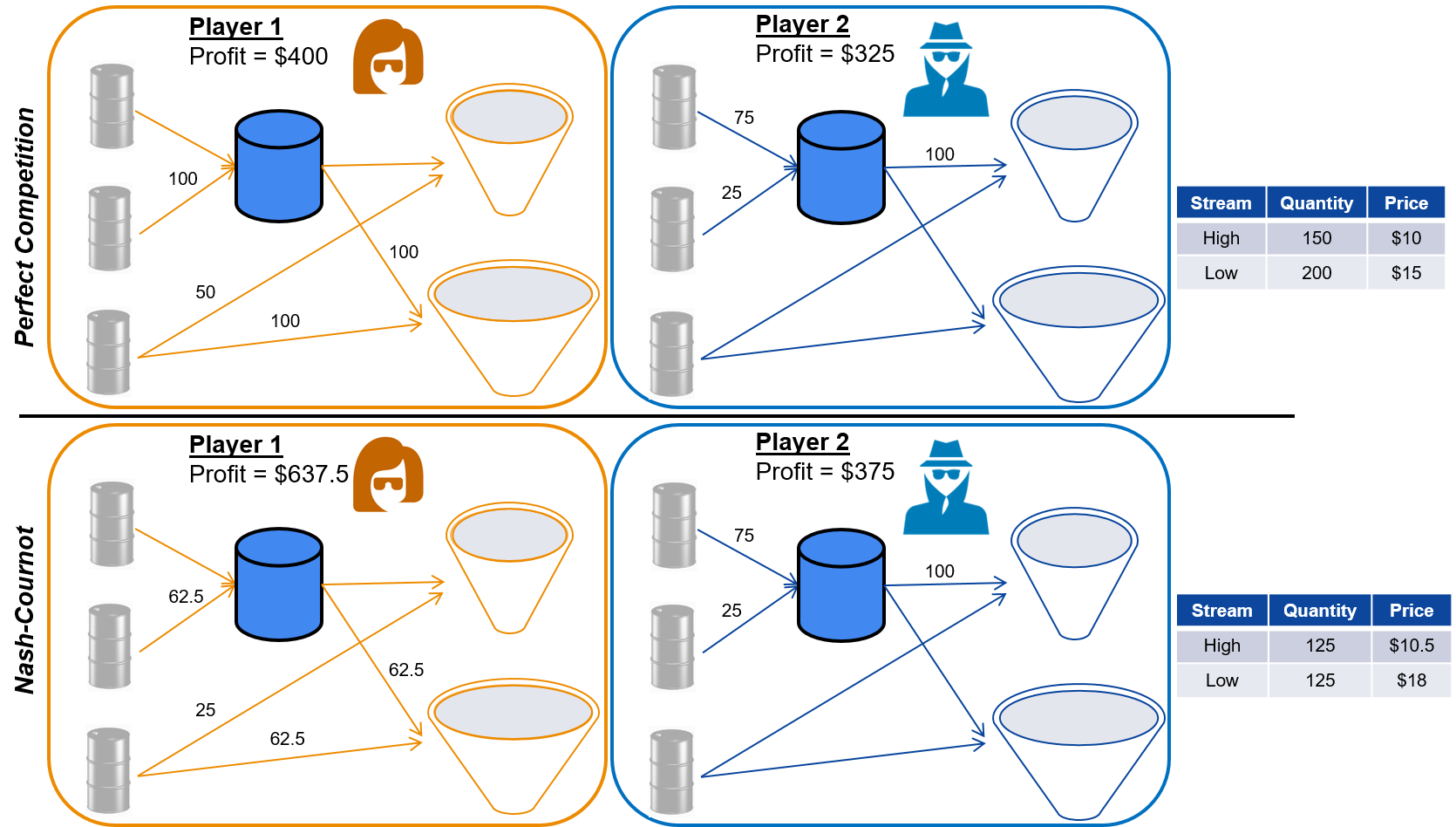}
\caption{Solution of our competitive two-player \textit{standard} (i.e., purely continuous) pooling problem instance.}
\label{fig:Competitive_Haverly_pooling_2player_solutions_6_16_10_v_3_18_11_0FixedCost}
\end{figure}

\begin{figure}[h!] 
\centering
\includegraphics[width=13cm]{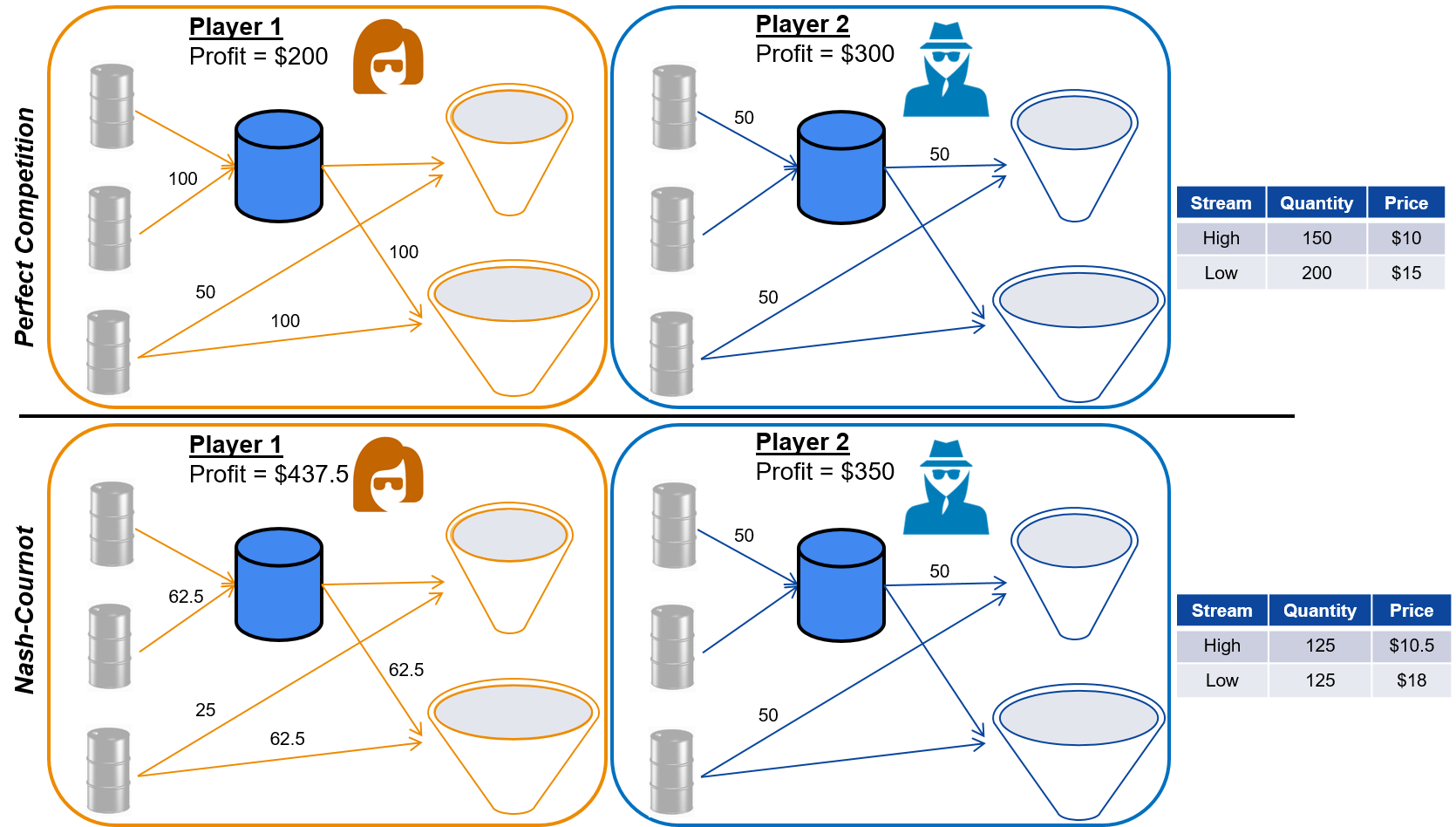}
\caption{Solution of our competitive two-player \textit{generalized} (i.e., mixed-integer) pooling problem instance.}
\label{fig:Competitive_Haverly_pooling_2player_solutions_6_16_10_v_3_18_11_1FixedCost}
\end{figure}

Several additional observations are noteworthy.  
First, a player's decisions may depend on the competitive setting assumed. 
Second, after introducing price elasticity and competition, neither player pursues the globally optimal solution to the original Haverly instance shown in Figure~\ref{fig:haverly_p}, which is to purchase 100 units of crudes B and C and produce 200 units of the low sulfur stream for a profit of \$400.   
Third, in the standard (purely continuous) pooling problem instance, player 1's profit is superior to player 2's in both competitive settings. This is no longer the case once a fixed cost is modeled as we see player 1's profit is less than player 2's in the perfect competition setting.  

For those familiar with the original Haverly problem, it is instructive to ask: 
Is the difference in the solution to the original Haverly instance in Figure~\ref{fig:haverly_p} and the standard pooling solutions shown in Figure~\ref{fig:Competitive_Haverly_pooling_2player_solutions_6_16_10_v_3_18_11_0FixedCost} due to the presence of a nonzero slope $\beta_n \neq 0$ for $n \in \Nout$ or to the introduction of competitive behavior? The answer is: It is due to both.  If the slope $\beta_n$ were 0 for $n \in \Nout$ (i.e., in economic speak, if demand remained perfectly elastic), the intercepts were identical to the original Haverly instance, i.e., $(\alpha_H,\alpha_L)=(9,15)$, and competition were introduced, then each player's optimal decision would be identical to that of the original Haverly instance in Figure~\ref{fig:haverly_p}. That is, the type of competition is irrelevant when the slopes are zero.  On the other hand, once the slopes are nonzero, the slopes and the type of competition assumed impact the players' decisions. 

\begin{figure}[h!] 
  \begin{subfigure}[b]{0.38\linewidth}
    \includegraphics[width=\textwidth]{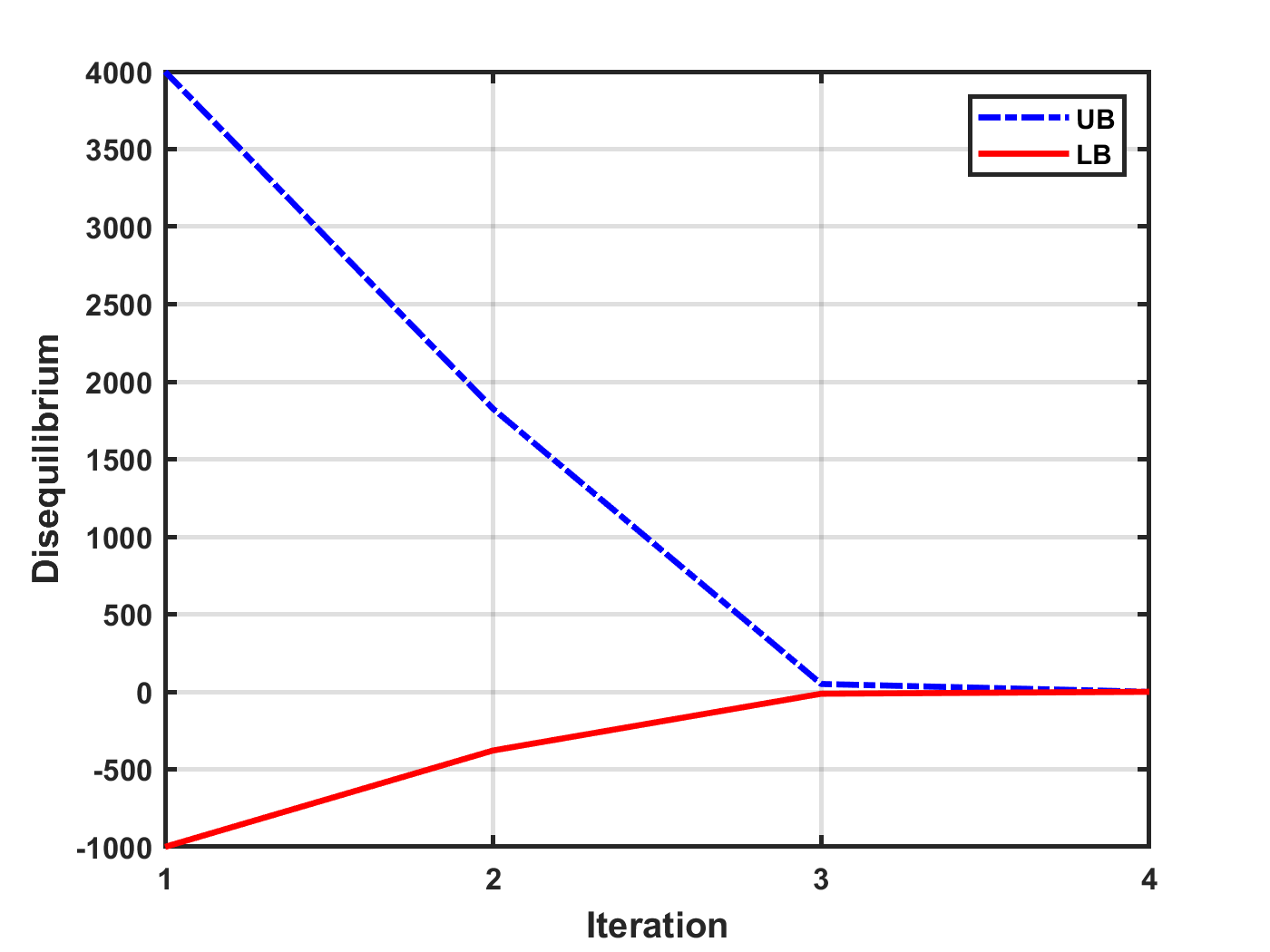}
    \caption{Convergence}
    \label{fig:1a}
  \end{subfigure}
  \begin{subfigure}[b]{0.38\linewidth}
    \includegraphics[width=\textwidth]{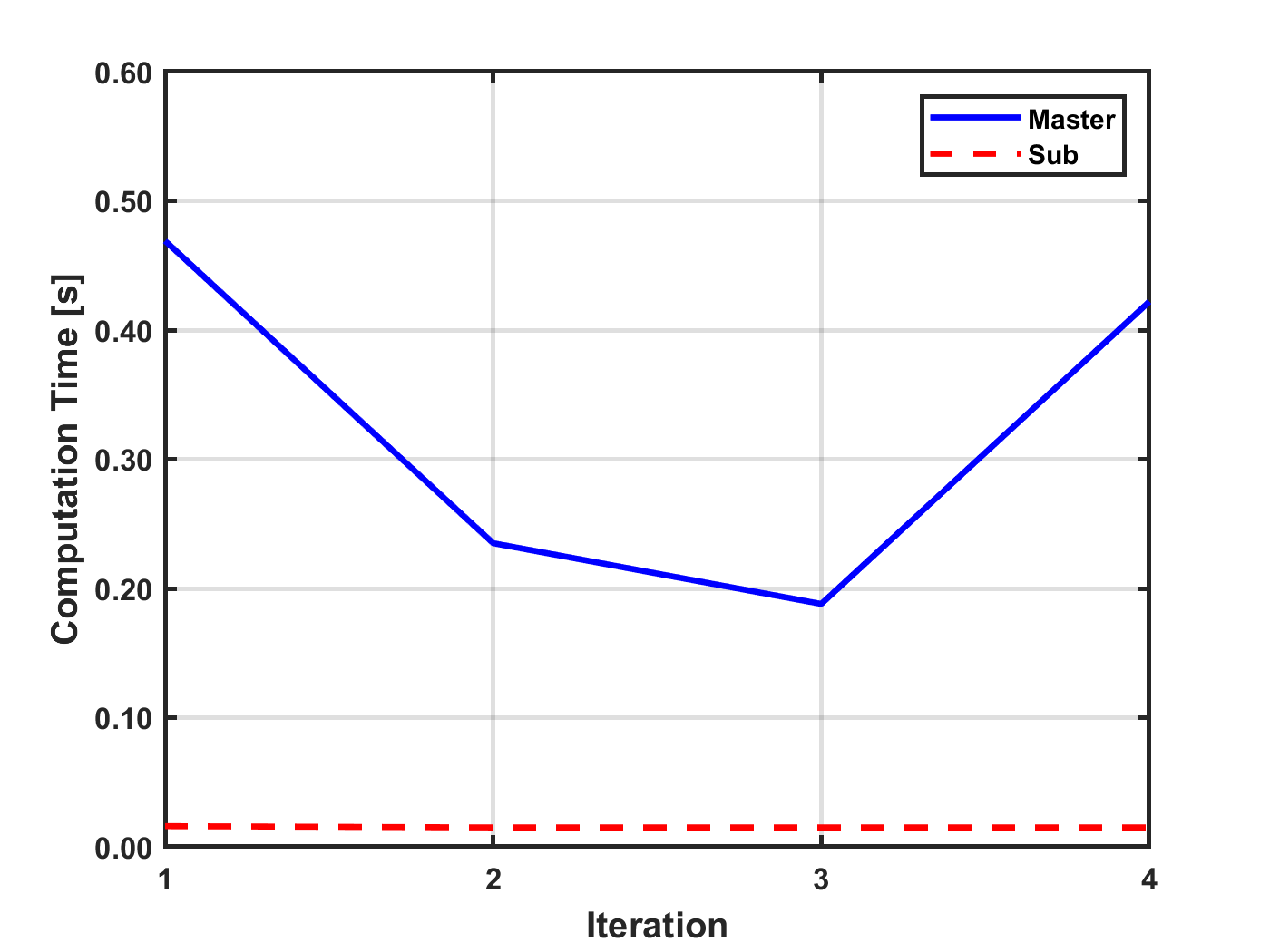}
    \caption{Solution times}
    \label{fig:1b}
  \end{subfigure}
  \caption{Algorithm performance for continuous perfect competition setting. (a) Upper bound (UB) $\eta^U$ and lower bound (LB) $\eta^L$ profiles of our minimum disequilibrium algorithm. (b) Solution times associated with solving the relaxed master problem (Master) and all individual player subproblems (Sub).}
  \label{fig:algo_performance_illustrative_example_pc}
\end{figure} 

\begin{figure}[h!] 
  \begin{subfigure}[b]{0.38\linewidth}
    \includegraphics[width=\textwidth]{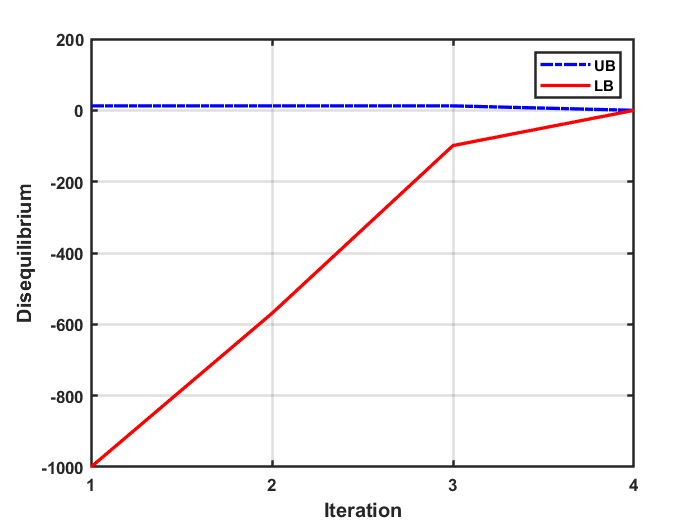}
    \caption{Convergence}
    \label{fig:2a_nc}
  \end{subfigure}
  \begin{subfigure}[b]{0.38\linewidth}
    \includegraphics[width=\textwidth]{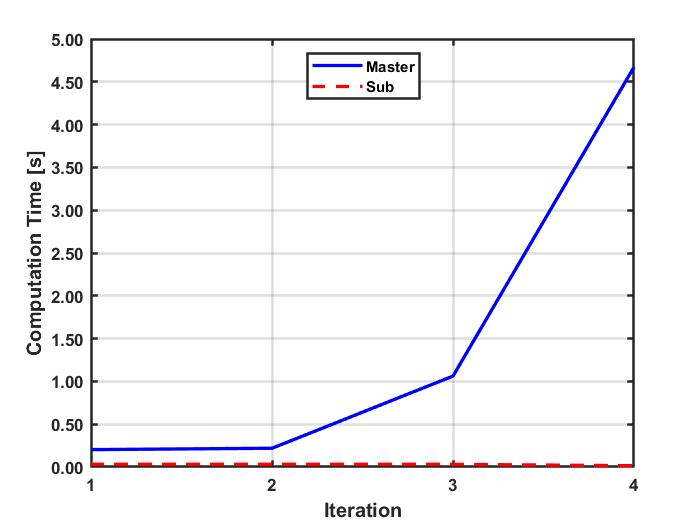}
    \caption{Solution times}
    \label{fig:2b_nc}
  \end{subfigure}
  \caption{Algorithm performance for continuous Nash-Cournot setting.}
  \label{fig:algo_performance_illustrative_example_nc}
\end{figure} 

We conclude this subsection by discussing algorithmic performance on these illustrative examples.
For these small instances, the heuristics from Section~\ref{sec:heuristics} produce an optimal solution (an equilibrium) in each case. 
Specifically, \cmtt{Gurobi 9.5} solves the resulting nonconvex MIQCQP instances of Model~\eqref{model:max_social_welfare} and Model~\eqref{model:relaxed_monolithic_NC} in roughly 0.03 seconds, while BARON 21 is closer to 0.55 seconds.
Meanwhile, Algorithm~\ref{algo:cut_generating_minimum_disequilibrium} solves the two continuous instances in four iterations and the two mixed-integer instances in three iterations. The results for the continuous instances are shown in Figure~\ref{fig:algo_performance_illustrative_example_pc} and Figure~\ref{fig:algo_performance_illustrative_example_nc}. Figure~\ref{fig:algo_performance_illustrative_example_nc}a suggests that a near-equilibrium (good upper bound) was found in the first iteration in the Nash-Cournot setting, unlike in Figure~\ref{fig:algo_performance_illustrative_example_pc}a in the perfect competition setting.
As for solution times, Figure~\ref{fig:algo_performance_illustrative_example_pc}b and  Figure~\ref{fig:algo_performance_illustrative_example_nc}b indicate that the aggregate time to solve all individual player subproblems in a single iteration of Algorithm~\ref{algo:cut_generating_minimum_disequilibrium} was between 0.01 and 0.03 seconds, on par with the amount of time required to solve the formulations \eqref{model:max_social_welfare} and \eqref{model:relaxed_monolithic_NC}.  On the other hand, the relaxed master problem \eqref{model:minimum_disequilibrium_relaxed}, a nonconvex QCQP, required between 0.2 and 0.5 seconds per solve in the perfect competition setting.  Figure~\ref{fig:algo_performance_illustrative_example_nc}b reveals that, in the Nash-Cournot setting, the solution times of the relaxed master problem \eqref{model:minimum_disequilibrium_relaxed} start small, but then reach 4.672 seconds in the last iteration of Algorithm~\ref{algo:cut_generating_minimum_disequilibrium}. 

\cmtt{
\subsection{When complementarity may fail to converge to an equilibrium}
\label{sec:comp_fails}
}
\cmtt{
As was underscored in Sections~\ref{sec:complementarity_heuristic} and \ref{sec:fallacy}, solving a complementarity model~\eqref{model:generic_nlp_kkt} or solving the monolithic Nash-Cournot model~\eqref{model:relaxed_monolithic_NC} for a nonconvex problem may fail to find an equilibrium because KKT conditions are only necessary, but not sufficient. 
It is instructive to observe this behavior with a concrete numerical example.
We now present a two-player standard (i.e., purely continuous) pooling instance with Nash-Cournot players in which an equilibrium exists, but the complementarity heuristic converges to a solution that is not an equilibrium. 
\begin{figure}[h]  
\centering
\includegraphics[width=13cm]{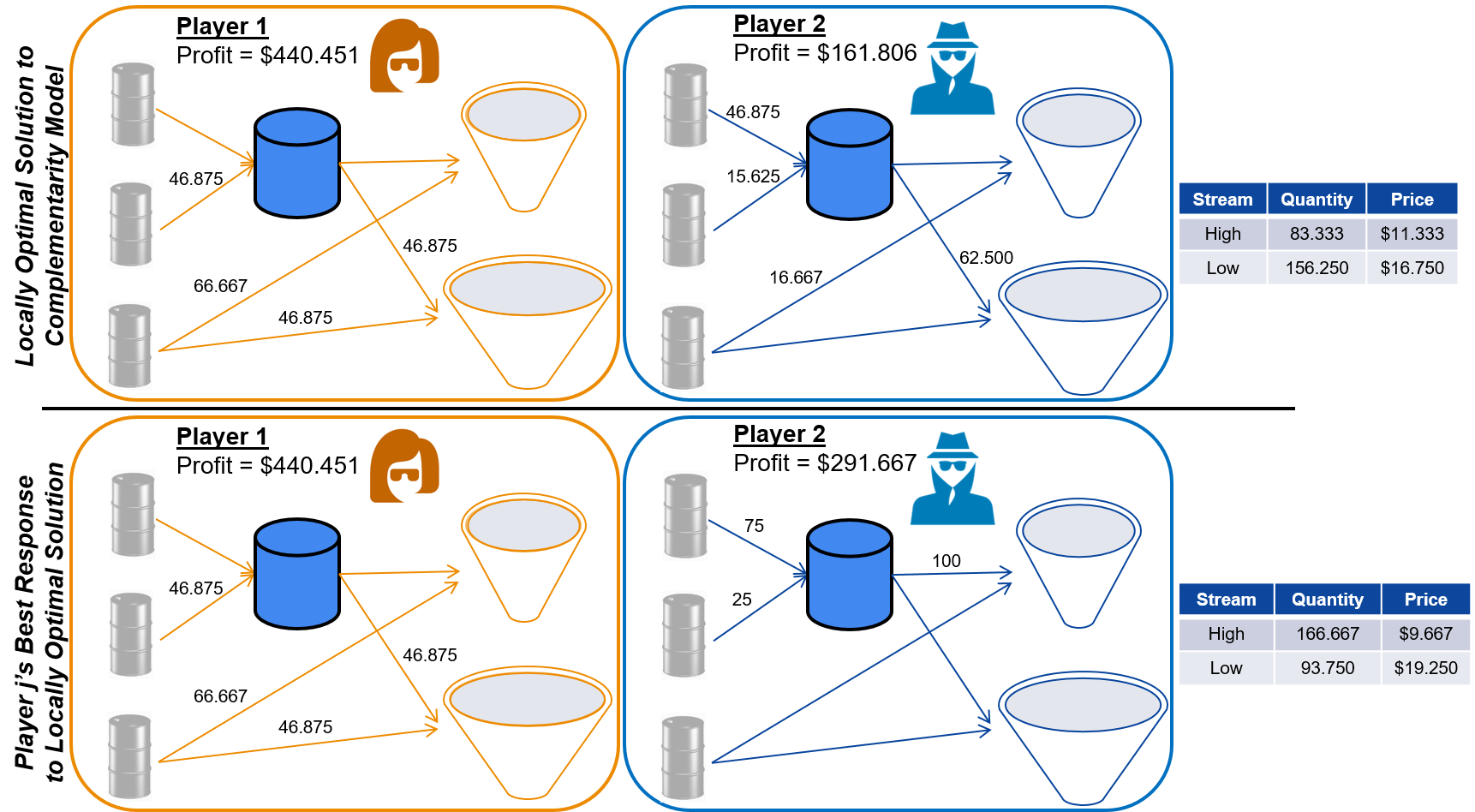}
\caption{Two-player Nash-Cournot instance for which an equilibrium exists, but solving monolithic Nash-Cournot model~\eqref{model:relaxed_monolithic_NC} to local optimality returns a solution that is \textit{not} a Nash equilibrium. Top row: Locally optimal solution (flows), profits, and prices obtained from solving  model~\eqref{model:relaxed_monolithic_NC} to local optimality with an initial solution of $\v{0}$. Bottom row: Player $j$'s optimal solution given that the other player's decisions are fixed to the values shown in the top row. Player 2 clearly has an incentive to deviate in order to earn a higher profit, illustrating that the locally optimal solution to model~\eqref{model:relaxed_monolithic_NC} cannot be a Nash equilibrium.}
\label{fig:complementarity_fails}
\end{figure} 
}

\cmtt{
Using the exact same parameter settings and conditions as shown in Figure~\ref{fig:Competitive_Haverly_pooling_2player_instance},
we attempt to find an equilibrium to this nonconvex Nash-Cournot game by solving monolithic Nash-Cournot model~\eqref{model:relaxed_monolithic_NC} to local optimality using an initial solution $(\v{x}_1,\v{x}_2)=(\v{0},\v{0})$, i.e., all decision variables are initialized to 0.
CONOPT 4.1 and IPOPT 3.11 both converge to the locally, but not globally, optimal solution shown in the top row of Figure~\ref{fig:complementarity_fails}. Both nonlinear solvers declare that the solution satisfies all KKT conditions and is therefore a locally optimal solution.
In this locally optimal solution, player 1 (player 2) earns a profit of roughly \$440 (\$162).
However, it is easy to show that this solution is \textit{not} an equilibrium.
Fixing each player's decision vector $\v{x}_{j}$ shown in the top row of Figure~\ref{fig:complementarity_fails}, one can determine the other player's optimal response and corresponding objective function value $p_j^*(\v{x}_{-j})$.  By definition, in a Nash equilibrium the objective function value $p_j^*(\v{x}_{-j})$ should not change (although the solution itself may change). The bottom row of Figure~\ref{fig:complementarity_fails} reveals that player 1's objective function value does not change (i.e., it stays at $\sim\$440$), whereas player 2 has an incentive deviate so that she can improve her profit from $\sim\$162$ to \$292. 
We re-iterate that, in our computational experiments, solving monolithic Nash-Cournot model~\eqref{model:relaxed_monolithic_NC} to global optimality did produce an equilibrium in all instances in which an equilibrium does, in fact, exist.  It is not clear to us if this result generalizes.
For a possible counterexample, see Appendix~\ref{app:counterexample}.
}

\subsection{What happens when no equilibrium exists?}
\label{sec:no_equilibrium}

Although rare, it is possible that no equilibrium exists.  In this situation, players always have an incentive to deviate given the decisions of all other players.  In this subsection, we showcase how our minimum disequilibrium approach is able to rigorously and systematically determine that no equilibrium exists, while returning the total disequilibrium (a strictly positive scalar) as a certificate of non-existence.  We compare this against the heuristics from Section~\ref{sec:heuristics}, which not only fail to provide a certificate of non-existence, but worse, fail to identify that no equilibrium exists. Note that a traditional fixed-point method \cmtt{(i.e., the Practitioner's method~\ref{algo:practitioners_method_Jacobi_type})} would also fail to provide this information.

\begin{figure}[h] 
\centering
\includegraphics[width=10cm]{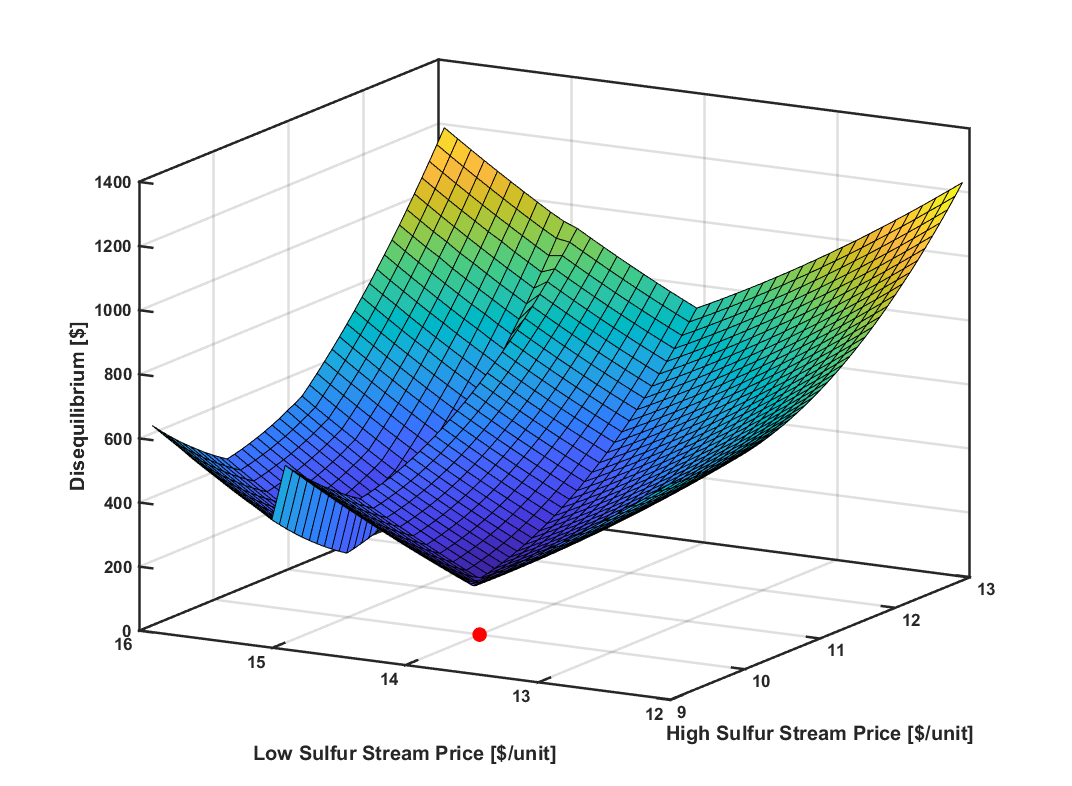}
\caption{Two-player perfect competition instance with no equilibrium. Both players have identical parameters and no fixed costs. Disequilibrium is minimized at a value of \$175 when the price of the high and low sulfur streams are \$10 and \$14 per unit, respectively.}
\label{fig:Competitive_Haverly_pooling_2player_PCinstance_0FixedCharge_NoEquilibrium}
\end{figure}

Consider a symmetric two-player perfect competition example in which both players have identical parameters equal to those shown for player 1 in Figure~\ref{fig:Competitive_Haverly_pooling_2player_instance}, i.e., variable cost $C_{jn}^{\textrm{var}}$ for crude A, B, and C is \$6, \$16, and \$10, respectively, for both players.  There are no fixed costs. The market parameters are those shown in Figure~\ref{fig:Competitive_Haverly_pooling_2player_instance}. Note that, in this example, we do not permit the consumer to have any disequilibrium.  Hence, we refer to this instance as a two-player (-supplier) perfect competition instance, when in fact it could also be treated as a three-player (two supplier, one consumer) instance.  

With this setup, Figure~\ref{fig:Competitive_Haverly_pooling_2player_PCinstance_0FixedCharge_NoEquilibrium} shows the disequilibrium summed over both players as a function of the price of each of the two output streams.  This figure was generated using the mesh grid approach described in Section~\ref{sec:mesh_grid} in which, after fixing the price vector $\vpi$ in 10 cent increments, one solves a restriction of the minimum disequilibrium model~\eqref{model:minimum_disequilibrium}.
In this example, disequilibrium is minimized when $(\pi_H,\pi_L)=(10,14)$, i.e., the price of the high (low) sulfur stream is \$10 (\$14) per unit. At this price vector, disequilibrium is \$175 and this disequilibrium is shared by the two identical suppliers; no disequilibrium is allowed to be allocated to the consumer. 
Figure~\ref{fig:Competitive_Haverly_pooling_2player_PCinstance_0FixedCharge_NoEquilibrium} is also interesting because it reveals potential discontinuities in the disequilibrium function.  

\begin{figure}[h] 
\centering
\includegraphics[width=8cm]{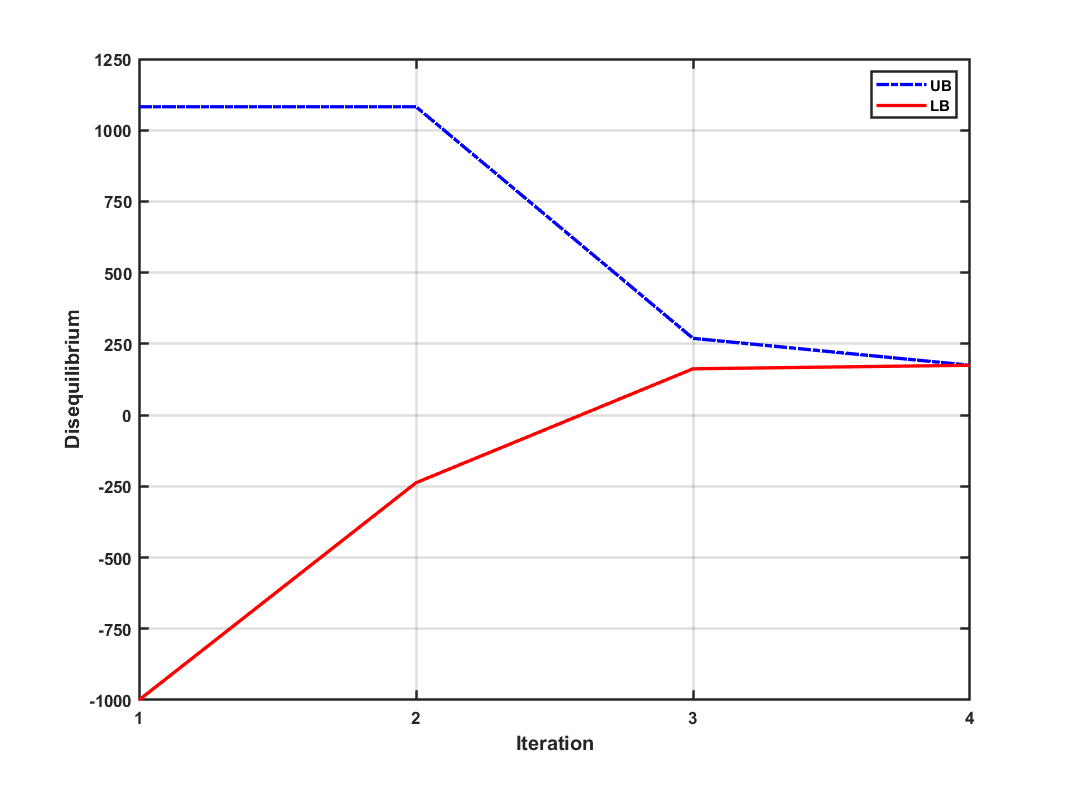}
\caption{Upper bound (UB) $\eta^U$ and lower bound (LB) $\eta^L$ profiles of our minimum disequilibrium algorithm on a two-player perfect competition instance with no equilibrium. Both players have identical parameters and no fixed costs. Disequilibrium is minimized at a value of 175, which is achieved in 4 iterations.}
\label{fig:Competitive_Haverly_pooling_2player_PCinstance_0FixedCharge_NoEquilibrium_bound_profiles}
\end{figure}

Using the heuristic in which we solve QCQP Model~\eqref{model:max_social_welfare} yields the solution shown for the perfect competition setting in Figure~\ref{fig:Competitive_Haverly_pooling_2player_solutions_6_16_10_v_3_18_11_0FixedCost}.  Player 1 makes a profit of \$400, while player 2 earns \$200.  However, this solution is not an equilibrium as should be intuitively clear.  The players have identical parameters and do not see what the other player does; as price takers, they only see prices.  Thus, given prices of \$10 and \$15 for the high and low sulfur streams, respectively, both players have the desire to make the same decisions to maximize profit. This leads to both players making decisions to earn a profit of \$400, showing that player 2 has an incentive to deviate.

\begin{figure}[h!] 
\centering
\includegraphics[width=8cm]{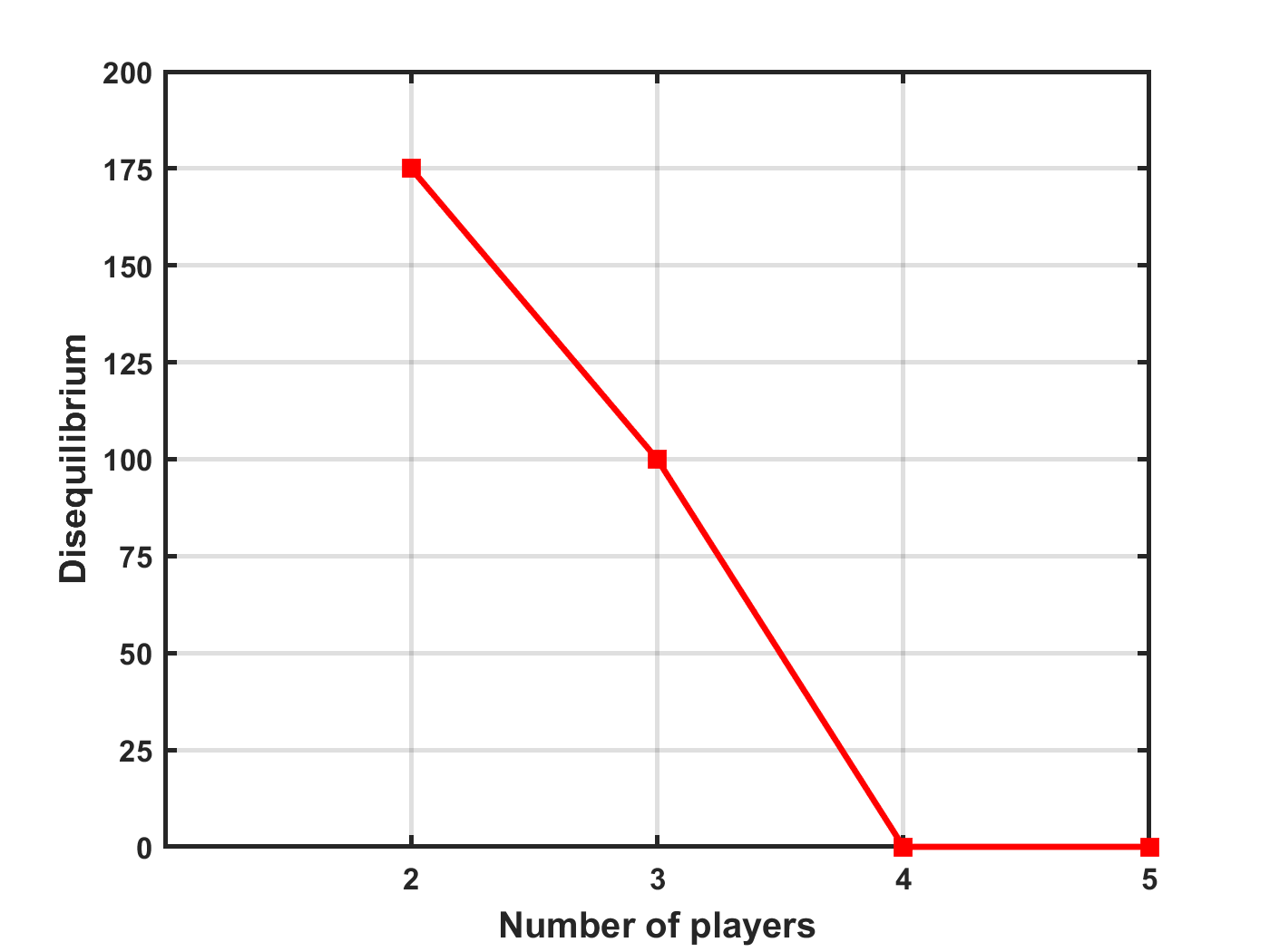}
\caption{Disequilibrium is eliminated in the perfect competition setting as the number of price-taking players $N$ increases.}
\label{fig:Disequilibrium_as_fnc_of_NumPlayers_PCexample}
\end{figure}   

In contrast to the heuristic, our approach of minimizing disequilibrium via Algorithm~\ref{algo:cut_generating_minimum_disequilibrium} provides a certificate (via the lower bound $\eta^L$) that no equilibrium exists.  Figure~\ref{fig:Competitive_Haverly_pooling_2player_PCinstance_0FixedCharge_NoEquilibrium_bound_profiles} shows the convergence of Algorithm~\ref{algo:cut_generating_minimum_disequilibrium} for this particular instance.  Algorithm~\ref{algo:cut_generating_minimum_disequilibrium} terminates in four iterations when the upper and lower bounds converge to a disequilibrium value of 175, proving that no solution with a disequilibrium less than 175 is possible. Note that if we were solely interested in proving that no equilibrium exists, the algorithm could have terminated at the end of iteration 3 when it found that $\eta_L \geq 162.5$.  For this example, however, we did not allow for early termination in order to find the optimum (minimum) disequilibrium. 

Incidentally, as shown in Figure~\ref{fig:Disequilibrium_as_fnc_of_NumPlayers_PCexample}, the minimum disequilibrium reduces to 100 when $N=3$ identical players are present.  When $N=4$, the minimum disequilibrium reduces further to 0 and an equilibrium is found.  This observation is not surprising as a perfect competition setting typically assumes that there are a large number of players each having no market power.

\cmtt{
\subsection{Algorithmic scale-up} \label{sec:algorithm_scaleup}
}
\cmtt{
Having demonstrated the fundamental differences between a standard single-player Haverly instance and game-theoretic multi-player Haverly-esque instances, we now demonstrate that Algorithm~\ref{algo:cut_generating_minimum_disequilibrium} is capable of handling larger instances.  
Subsection~\ref{sec:algorithm_scaleup_experimental_setup} outlines our experiments and metrics involving multi-quality instances. 
Subsection~\ref{sec:algorithm_scaleup_MD_performance} showcases the performance of our minimum disequilibrium algorithm with the simple warmstart procedures described in Subsection~\ref{sec:practitioners_method}.
Subsection~\ref{sec:algorithm_scaleup_complementarity_performance} compares our approach against arguably the state-of-the-art heuristic for solving equilibrium problems. 
}

\cmtt{
\subsubsection{Experimental setup with multi-quality instances} \label{sec:algorithm_scaleup_experimental_setup}
}
\cmtt{
Since Baltean-Lugojan \& Misener \cite{baltean2018piecewise} show that standard pooling problems with a ``single quality standard'' can be solved in strongly polynomial time, we consider two additional multi-quality instances -- Adhya1 and Bental5 -- to show how our algorithm scales for settings in which the individual player subproblems cannot be solved in polynomial time when no integer decisions are present. We selected these instances for the following reasons: 
(1) Adhya1\footnote{\url{www.ii.uib.no/~mohammeda/spooling/Adhya1.gms} The single-player instance (subproblem) has $|\Nin|=5$ input nodes, $|\Nblend|=2$ pools, $|\Nout|=4$ output nodes, and $|\mc{K}|=4$ specs. The $\mathbb{PQ}$ formulation has 33 variables, 20 distinct bilinear terms, and 42 linear constraints.} is at least three times larger (in terms of variables and constraints) than the standard Haverly instance. 
(2) Bental5\footnote{\url{www.ii.uib.no/~mohammeda/spooling/Bental5.gms} The single-player instance (subproblem) has $|\Nin|=5$ input nodes, $|\Nblend|=3$ pools, $|\Nout|=5$ output nodes, and $|\mc{K}|=2$ specs. The $\mathbb{PQ}$ formulation has 92 variables, 60 distinct bilinear terms, and 53 linear constraints.} is the largest multi-quality ``standard test instance'' (in terms of variables and constraints) considered in Alfaki and Haugland \cite[Table 1]{alfaki2013strong}.
For each of these instances, we investigate two forms of player behavior (price-taker and Nash-Cournot) and two model variants (nonconvex QCQP, i.e., those without fixed costs, and nonconvex MIQCQP, i.e., those with fixed costs). 
Thus, in total, we consider 8 multi-quality two-player noncooperative game instances. The complementarity model for these 8 instances, as well as the 4 Haverly-esque instances above, are publicly-available on \url{minlp.org}.
}

\begin{table}[h!]													
\centering	
{\color{black}													
\begin{tabular}{crrccccccc}													
\toprule													
	&	\multicolumn{2}{c}{Variable cost}			&	\multicolumn{2}{c}{Fixed cost}	&&& \multicolumn{2}{c}{Inverse demand fnc}			\\
		\cmidrule(lr){2-3}								\cmidrule(lr){4-5}							\cmidrule(lr){8-9}			
$n \in \Nin$	&	$j=1$	&	$2$	&	$1$	&	$2$	&&	$n \in \Nout$	&	$\alpha_n$	&	$\beta_n$	\\
\midrule													
1	&	7	&	6	&	100	&	100	&&	8	&	16	&	0	\\
2	&	3	&	4	&	100	&	100	&&	9	&	32	&	0.28	\\
3	&	2	&	1	&	100	&	100	&&	10	&	15	&	0	\\
4	&	10	&	9	&	100	&	100	&&	11	&	12	&	0.20	\\
5	&	5	&	7	&	100	&	100	&&		&		&		\\
\bottomrule													
\end{tabular}													
\caption{Modifications to the original Adhya1 instance for our two-player game setting. Player 1's variable costs are identical to those given in the original Adhya1 instance.}													
\label{tbl:Adhya1}
} 
\end{table}													

\begin{table}[h!]													
\centering	
{\color{black}													
\begin{tabular}{crrccccccc}													
\toprule													
	&	\multicolumn{2}{c}{Variable cost}			&	\multicolumn{2}{c}{Fixed cost}	&&& \multicolumn{2}{c}{Inverse demand fnc}			\\
		\cmidrule(lr){2-3}								\cmidrule(lr){4-5}							\cmidrule(lr){8-9}			
$n \in \Nin$	&	$j=1$	&	$2$	&	$1$	&	$2$	&&	$n \in \Nout$	&	$\alpha_n$	&	$\beta_n$	\\
\midrule													
1	&	6	&	7	& 200 & 200	&	&	9	&	25	&	0.07	\\
2	&	16	&	14	& 200 & 200	&	&	10	&	23	&	0.04	\\
3	&	15	&	13	& 200 & 200	&	&	11	&	27	&	0.08	\\
4	&	12	&	15	& 200 & 200	&	&	12	&	22	&	0.06	\\
5	&	10	&	10	& 200 & 200	&	&	13	&	17	&	0.03	\\
\bottomrule													
\end{tabular}													
\caption{Modifications to the original Bental5 instance for our two-player game setting. Player 1's variable costs are identical to those given in the original Bental5 instance.}													
\label{tbl:Bental5}	
} 
\end{table}

\cmtt{
To convert the single-player Adhya1 and Bental5 instances into two-player noncooperative games, we assume that both players have the same network structure and parameters as the original single-player instance, but with three important modifications. First, we assume that the second player's variable costs are different from those of the first player (the original values).  Second, we introduce a fixed cost for purchasing any positive amount of a raw material. Third, we replace the fixed prices associated with each output node with independent linear inverse demand functions, just as we did in our Haverly-esque example.  The variable costs, fixed costs, and inverse demand function parameters are listed in Tables~\ref{tbl:Adhya1} and \ref{tbl:Bental5}.   
}

\cmtt{
Before delving into the results, we must first describe our experimental setup. 
Table~\ref{tbl:Table_Key} explains the metrics used to assess algorithmic performance within this setup.
As stated above, each base instance (Adhya1 and Bental5) gives rise to four instances depending on the form of competition (perfect or imperfect) and the presence of binary decisions (QCQP or MIQCQP players).
Since it is customary to warmstart the RMP~\eqref{model:minimum_disequilibrium_relaxed} with some initial solutions (in $\mc{X}_j^L$) to accelerate Algorithm~\ref{algo:cut_generating_minimum_disequilibrium}, especially in larger-scale instances, for each instance (e.g., row in Table~\ref{tbl:Adhya1} or Table~\ref{tbl:Bental5}), we consider at least five random seeds to generate these initial solutions.  These random seeds affect the uniform random variables appearing in warmstart Algorithm~\ref{algo:warmstart_method_Price_Taker} and Algorithm~\ref{algo:warmstart_method_Nash_Cournot_Gauss_Seidel_type}.  
We chose a random initialization to avoid giving the impression that we fine-tuned our warmstart procedure.
Since we will later compare the performance of our minimum disequilibrium algorithm with complementarity-based heuristics, we chose to warmstart with a variant of the Practitioner's algorithm, not the complementarity heuristic. 
Without this warmstart, in early iterations, the RMP yields ``unproductive'' candidate solutions, which in our setting means solutions that are far from being equilibria. The astute reader will recognize that the same behavior occurs in a Benders decomposition algorithm with no initial optimality or feasibility cuts. 
We report the number of warmstart solutions generated in the column ``WS Cuts'' because these solutions are used as cuts in the RMP~\eqref{model:minimum_disequilibrium_relaxed}. We also report the time spent in the warmstart (``WS'' column) phase, solving the RMP (``Master'' column), and subproblems (``Sub'' column) associated with Algorithm~\ref{algo:cut_generating_minimum_disequilibrium}.  To be clear, any individual player subproblem solved during the warmstart phase is only counted in the ``WS'' column, not in the ``Sub'' column.
}

\begin{table}[th!]			
\centering			
{\color{black}
\begin{tabular}{ll}			
\toprule			
Column	&	Explanation	\\
\midrule			
Comp	&	Competition type: PT = Price Taker, NC = Nash-Cournot	\\
Int	&	Does the instance include integer decisions? 0 = No, the subproblems~\eqref{model:agent_problem_Nash} and the RMP~\eqref{model:minimum_disequilibrium_relaxed} \\ 
	& \quad are nonconvex QCQPs; otherwise (1 = Yes), they are nonconvex MIQCQPs	\\
Seed	&	Random seed (an integer) used for warmstart phase; ``NA'' means no warmstart was used 	\\
$\eta^{L}$	&	Lower bound on total disequilibrium (see Algorithm~\ref{algo:cut_generating_minimum_disequilibrium})	\\
$\eta^{U}$	&	Upper bound on total disequilibrium	(see Algorithm~\ref{algo:cut_generating_minimum_disequilibrium}) \\
Iters	&	Total number of major iterations executed in Algorithm~\ref{algo:cut_generating_minimum_disequilibrium}; \\
		& \quad 1 major iteration = 1 pass through the \textbf{while} loop	\\
PBF	&	Phase Best Found: Phase (WS=Warmstart, S=Standard) \\
	& \quad in which best solution (equilibrium) was found 	\\
IBF	&	Iteration Best solution was Found. A positive integer $k$ means the solution was found in the \\
& \quad $k$th major iteration of Algorithm~\ref{algo:cut_generating_minimum_disequilibrium}; 0 implies the solution was found in the warmstart phase 	\\
$p_j^*$	&	Profit found for player $j$ in the ``best'' (closest to equilibrium) solution found	\\
RGAP $\delta_j^*$\%	& Relative disequilibrium gap for player $j$:	$100\% ( p_j^U(\hat{\vpi}) - p_j(\hat{\v{x}}_j,\hat{\vpi}) )/p_j^U(\hat{\vpi})$	\\
	& \quad where $p_j^U(\hat{\vpi})$ is the upper (dual) bound on $p_j^*(\hat{\vpi})$. See also definition \eqref{eq:relative_disequilibrium_gap_definition}  \\
WS Cuts	&	Number of warmstart cuts / solutions generated in $\mc{X}_j^U$ for each player $j$	\\
WS (Time)	&	Time [s] spent solving subproblems in the warmstart phase	\\
Master (Time)	&	Time [s] spent solving the relaxed master problem~\eqref{model:minimum_disequilibrium_relaxed}	\\
Sub (Time)	&	Time [s] spent solving subproblems in Algorithm~\ref{algo:cut_generating_minimum_disequilibrium} Step~\ref{step:solve_subproblems}	\\
Total (Time)	&	Time [s] spent in all phases \\
\bottomrule			
\end{tabular}			
\caption{Explanation of columns in Tables~\ref{tbl:Adhya1_results}, \ref{tbl:Bental5_results}, and \ref{tbl:Complementarity_results}.}			
\label{tbl:Table_Key}			
} 
\end{table}		

\cmtt{
In Tables~\ref{tbl:Adhya1_results}, \ref{tbl:Bental5_results}, and \ref{tbl:Complementarity_results}, we  report the profit associated with the best solution found for each player during the course of the algorithm and its associated relative disequilibrium gap. Here ``best'' is defined as a solution resulting in minimum disequilibrium.  
Since a player's absolute disequilibrium \eqref{eq:disequilibrium_def} is scale dependent,
we prefer to use relative disequilibrium gap
\begin{equation} \label{eq:relative_disequilibrium_gap_definition}
\textrm{RGAP } \delta_j^* = \frac{ p_j^U(\vpi^*) - p_j(\v{x}_j^*,\vpi^*) }{p_j^U(\vpi^*)}
\end{equation}
as a metric to measure the disequilibrium in player $j$'s solution relative to the upper bound $p_j^U(\vpi^*)$ on her optimal profit $p_j^*(\vpi^*)$.
If $(\v{x}_j^*,\vpi^*)$ is an equilibrium, but it is computational expensive to solve the subproblem to global optimality for $p_j^*(\vpi^*)$, then we use the dual bound $p_j^U(\vpi^*)$ instead in the relative disequilibrium gap calculation.  
If the subproblem is solved to provable optimality, then $p_j^U(\vpi^*)=p_j^*(\vpi^*)$ and we recover the standard absolute disequilibrium in the numerator.	
}

\afterpage{%
    \clearpage
    \thispagestyle{empty}
    \begin{landscape}
\begin{table}																									
\centering	
{\color{black}																										
\begin{tabular}{cccrrrcrrrrrrrrrrr}																																			
\toprule																																			
Comp	&	Int	&	Seed	&	$\eta^{LB}$	&	$\eta^{UB}$	&	Iters	&	PBF	&	IBF	&	\multicolumn{2}{c}{$p_j^*$}			&	\multicolumn{2}{c}{RGAP $\delta_j^*$\%}			&	\multicolumn{2}{c}{WS Cuts}			&	\multicolumn{4}{c}{Solution Time [s]}							\\
																\cmidrule(lr){9-10}				\cmidrule(lr){11-12}				\cmidrule(lr){13-14}				\cmidrule(lr){15-18}							
	&		&		&		&		&		&		&		&	1	&	2	&	1	&	2	&	1	&	2	&	WS	&	Master	&	Sub	&	Total	\\
\midrule																																			
PT	&	0	&	1	&	-21.94	&	1.86	&	1	&	S	&	1	&	354.81	&	344.38	&	0.15	&	0.39	&	4	&	5	&	32	&	600	&	45	&	677	\\
PT	&	0	&	2	&	-22.31	&	2.18	&	1	&	S	&	1	&	354.80	&	344.37	&	0.58	&	0.03	&	4	&	4	&	59	&	600	&	52	&	711	\\
PT	&	0	&	3	&	-22.06	&	0.37	&	1	&	S	&	1	&	354.80	&	344.37	&	0.04	&	0.06	&	5	&	3	&	41	&	600	&	51	&	692	\\
PT	&	0	&	4	&	-22.01	&	1.20	&	1	&	S	&	1	&	354.92	&	344.49	&	0.08	&	0.27	&	4	&	5	&	47	&	600	&	53	&	699	\\
PT	&	0	&	5	&	-22.49	&	1.33	&	1	&	S	&	1	&	354.80	&	344.37	&	0.14	&	0.24	&	2	&	3	&	40	&	600	&	57	&	697	\\
PT	&	0	&	NA	&	-28.67	&	2.15	&	10	&	S	&	3	&	354.82	&	344.39	&	0.52	&	0.09	&		&		&		&	1776	&	250	&	2027	\\
\midrule		
PT	&	1	&	1	&	0.00	&	0.00	&	1	&	S	&	1	&	87.50	&	97.50	&	0.00	&	0.00	&	4	&	4	&	7	&	12	&	3	&	22	\\
PT	&	1	&	2	&	0.00	&	0.00	&	1	&	S	&	1	&	87.50	&	97.50	&	0.00	&	0.00	&	4	&	4	&	8	&	13	&	2	&	23	\\
PT	&	1	&	3	&	0.00	&	0.00	&	1	&	S	&	1	&	87.50	&	97.50	&	0.00	&	0.00	&	4	&	4	&	8	&	12	&	2	&	21	\\
PT	&	1	&	4	&	0.00	&	0.00	&	1	&	S	&	1	&	87.50	&	97.50	&	0.00	&	0.00	&	4	&	4	&	7	&	13	&	2	&	22	\\
PT	&	1	&	5	&	0.00	&	0.00	&	1	&	S	&	1	&	87.50	&	97.50	&	0.00	&	0.00	&	4	&	4	&	6	&	10	&	1	&	17	\\
PT	&	1	&	NA	&	0.00	&	0.00	&	3	&	S	&	3	&	87.50	&	97.50	&	0.00	&	0.00	&		&		&		&	25	&	3	&	28	\\
\midrule																											
NC	&	0	&	1	&		&		&	0	&	WS	&	1	&	354.80	&	344.37	&	0.00	&	0.00	&	1	&	1	&	87	&		&		&	87	\\
NC	&	0	&	2	&		&		&	0	&	WS	&	1	&	354.80	&	344.37	&	0.00	&	0.00	&	1	&	1	&	100	&		&		&	100	\\
NC	&	0	&	3	&		&		&	0	&	WS	&	1	&	354.80	&	344.37	&	0.00	&	0.00	&	1	&	1	&	94	&		&		&	94	\\
NC	&	0	&	4	&		&		&	0	&	WS	&	1	&	354.80	&	344.37	&	0.00	&	0.00	&	1	&	1	&	91	&		&		&	91	\\
NC	&	0	&	5	&		&		&	0	&	WS	&	1	&	354.80	&	344.37	&	0.00	&	0.00	&	1	&	1	&	94	&		&		&	94	\\
NC	&	0	&	NA	&	-55.55	&	2.90	&	3	&	S	&	2	&	354.80	&	344.37	&	0.39	&	0.44	&		&		&		&	540	&	208	&	748	\\
\midrule																																			
NC	&	1	&	1	&		&		&	0	&	WS	&	1	&	87.81	&	96.95	&	0.00	&	0.00	&	1	&	1	&	1	&		&		&	1	\\
NC	&	1	&	2	&		&		&	0	&	WS	&	1	&	87.81	&	96.96	&	0.00	&	0.00	&	1	&	1	&	1	&		&		&	1	\\
NC	&	1	&	3	&		&		&	0	&	WS	&	1	&	87.81	&	96.95	&	0.00	&	0.00	&	1	&	1	&	1	&		&		&	1	\\
NC	&	1	&	4	&		&		&	0	&	WS	&	1	&	87.81	&	96.95	&	0.00	&	0.00	&	1	&	1	&	1	&		&		&	1	\\
NC	&	1	&	5	&		&		&	0	&	WS	&	1	&	87.81	&	96.95	&	0.00	&	0.00	&	1	&	1	&	1	&		&		&	1	\\
NC	&	1	&	NA	&	0.00	&	0.00	&	10	&	S	&	10	&	88.05	&	96.39	&	0.00	&	0.00	&		&		&		&	349	&	7	&	356	\\
\bottomrule																		
\end{tabular}	
\caption{Adhya1-based two-player game results using our minimum disequilibrium algorithm with a naive practitioner's method for warmstart solutions. See Table~\ref{tbl:Table_Key} for an explanation of the columns.}			
\label{tbl:Adhya1_results}	
} 
\end{table}																									
\end{landscape}
    \clearpage
}

\afterpage{%
    \clearpage
    \thispagestyle{empty}
    \begin{landscape}
\begin{table}										
\centering		
{\color{black}																							
\begin{tabular}{cccrrrcrrrrrrrrrrr}																			
\toprule																									
Comp	&	Int	&	Seed	&	$\eta^{LB}$	&	$\eta^{UB}$	&	Iters	&	PBF	&	IBF	&	\multicolumn{2}{c}{$p_j^*$}			&	\multicolumn{2}{c}{RGAP $\delta_j^*$\%}			&	\multicolumn{2}{c}{WS Cuts}			&	\multicolumn{4}{c}{Solution Time [s]}							\\
																\cmidrule(lr){9-10}				\cmidrule(lr){11-12}				\cmidrule(lr){13-14}				\cmidrule(lr){15-18}						
	&		&		&		&		&		&		&		&	1	&	2	&	1	&	2	&	1	&	2	&	WS	&	Master	&	Sub	&	Total	\\
\midrule																																			
PT	&	0	&	1	&	-35.00	&	9.54	&	1	&	S	&	1	&	600.00	&	475.00	&	0.00	&	1.97	&	39	&	38	&	186	&	30	&	120	&	336	\\
PT	&	0	&	2	&	-34.80	&	9.07	&	1	&	S	&	1	&	600.00	&	475.00	&	0.00	&	1.87	&	45	&	36	&	178	&	30	&	120	&	328	\\
PT	&	0	&	3	&	-33.54	&	7.95	&	1	&	S	&	1	&	600.00	&	475.00	&	0.00	&	1.65	&	37	&	36	&	158	&	30	&	120	&	308	\\
PT	&	0	&	4	&	-35.00	&	8.34	&	1	&	S	&	1	&	600.00	&	475.00	&	0.00	&	1.73	&	36	&	36	&	195	&	30	&	120	&	345	\\
PT	&	0	&	5	&	-35.00	&	9.51	&	1	&	S	&	1	&	600.00	&	475.00	&	0.00	&	1.96	&	39	&	36	&	170	&	30	&	120	&	320	\\
PT	&	0	&	NA	&	-34.29	&	14.50	&	5	&	S	&	4	&	600.00	&	475.00	&	0.00	&	2.96	&		&		&		&	121	&	150	&	271	\\
\midrule																			
PT	&	1	&	1	&	9.29	&	9.33	&	1	&	S	&	1	&	275.00	&	0.00	&	0.00	&	0/0	&	17	&	19	&	30	&	0	&	4	&	34	\\
PT	&	1	&	2	&	9.29	&	9.33	&	1	&	S	&	1	&	275.00	&	0.00	&	0.00	&	0/0	&	18	&	19	&	30	&	0	&	3	&	33	\\
PT	&	1	&	3	&	3.13	&	9.33	&	1	&	S	&	1	&	275.00	&	0.00	&	0.00	&	0/0	&	20	&	17	&	30	&	0	&	3	&	34	\\
PT	&	1	&	4	&	9.27	&	9.33	&	1	&	S	&	1	&	275.00	&	0.00	&	0.00	&	0/0	&	18	&	20	&	30	&	0	&	3	&	34	\\
PT	&	1	&	5	&	9.26	&	9.33	&	2	&	S	&	1	&	275.00	&	0.00	&	0.00	&	0/0	&	15	&	16	&	60	&	0	&	3	&	63	\\
PT	&	1	&	NA	&	0.47	&	9.33	&	4	&	S	&	4	&	275.00	&	0.00	&	0.00	&	0/0	&		&		&		&	770	&	1	&	771	\\
\midrule																		
NC	&	0	&	1	&	-310.69	&	14.02	&	1	&	WS	&	10	&	1315.10	&	1764.19	&	0.00	&	0.79	&	10	&	10	&	203	&	30	&	10	&	243	\\
NC	&	0	&	2	&	-423.00	&	13.02	&	1	&	WS	&	1	&	1564.99	&	1643.14	&	0.00	&	0.79	&	10	&	10	&	202	&	30	&	10	&	243	\\
NC	&	0	&	3	&	-475.94	&	8.75	&	1	&	WS	&	1	&	1264.83	&	1794.99	&	0.00	&	0.49	&	10	&	10	&	603	&	30	&	10	&	643	\\
NC	&	0	&	4	&	-126.80	&	4.47	&	5	&	WS	&	6	&	1104.12	&	1890.38	&	0.00	&	0.23	&	10	&	10	&	3604	&	3000	&	902	&	7506	\\
NC	&	0	&	5	&	-120.00	&	4.59	&	5	&	WS	&	3	&	1104.09	&	1890.39	&	0.00	&	0.24	&	10	&	9	&	3605	&	3000	&	901	&	7506	\\
NC	&	0	&	6	&	-111.16	&	4.67	&	5	&	WS	&	9	&	1108.76	&	1887.17	&	0.01	&	0.24	&	9	&	9	&	3603	&	3000	&	902	&	7506	\\
NC	&	0	&	7	&	-127.13	&	4.90	&	5	&	WS	&	6	&	1176.82	&	1844.51	&	0.00	&	0.26	&	10	&	10	&	3604	&	3000	&	903	&	7507	\\
NC	&	0	&	8	&	-128.21	&	4.76	&	5	&	WS	&	7	&	1349.77	&	1743.59	&	0.00	&	0.27	&	10	&	10	&	3603	&	3000	&	902	&	7504	\\
NC	&	0	&	NA	&	-110.79	&	26.99	&	10	&	S	&	10	&	1738.71	&	1565.84	&	0.68	&	0.96	&		&		&		&	3000	&	904	&	3904	\\
\midrule																							
NC	&	1	&	1	&		&		&	0	&	WS	&	1	&	704.12	&	1356.79	&	0.00	&	0.00	&	1	&	1	&	5	&		&		&	5	\\
NC	&	1	&	2	&		&		&	0	&	WS	&	1	&	1027.45	&	1185.32	&	0.00	&	0.00	&	1	&	1	&	5	&		&		&	5	\\
NC	&	1	&	3	&		&		&	0	&	WS	&	1	&	857.18	&	1269.58	&	0.00	&	0.00	&	1	&	1	&	5	&		&		&	5	\\
NC	&	1	&	4	&		&		&	0	&	WS	&	1	&	1190.47	&	1098.24	&	0.00	&	0.00	&	1	&	1	&	5	&		&		&	5	\\
NC	&	1	&	5	&		&		&	0	&	WS	&	1	&	1039.65	&	1163.62	&	0.00	&	0.00	&	1	&	1	&	5	&		&		&	5	\\
NC	&	1	&	NA	&	-338.53	&	5.33	&	20	&	S	&	20	&	1286.81	&	1074.85	&	0.19	&	0.27	&		&		&		&	183	&	12	&	195	\\
NC	&	1	&	NA	&	-143.03	&	1.10	&	20	&	S	&	14	&	1408.71	&	970.24	&	0.05	&	0.04	&		&		&		&	1083	&	13	&	1096	\\
\bottomrule																					
\end{tabular}																								
\caption{Bental5-based two-player game results using our minimum disequilibrium algorithm with a naive practitioner's method for warmstart solutions. See Table~\ref{tbl:Table_Key} for an explanation of the columns.}				
\label{tbl:Bental5_results}			
} 
\end{table}																				
\end{landscape}
    \clearpage
}
\cmtt{
\subsubsection{Minimum disequilibrium algorithm performance} \label{sec:algorithm_scaleup_MD_performance}
}
\cmtt{
Tables~\ref{tbl:Adhya1_results} and \ref{tbl:Bental5_results} show the performance of our minimum disequilibrium method (Algorithm~\ref{algo:cut_generating_minimum_disequilibrium}) with warmstart Algorithm~\ref{algo:warmstart_method_Price_Taker} (in the price-taker setting),  Algorithm~\ref{algo:warmstart_method_Nash_Cournot_Gauss_Seidel_type} (in the Nash-Cournot setting),
and no warmstart whatsoever. 
We discuss these results in the paragraphs that follow.
Several take-aways are noteworthy.  
First, our minimum disequilibrium algorithm once again certifies that no equilibrium exists in the price-taker MIQCQP setting.
Second, QCQP subproblems require more computational time to converge and sometimes reach the time limit allotted without fully converging to optimality. We believe that a stronger pooling formulation could remedy this slow convergence.
Third, the single-iteration warmstart methods accelerate the time to find an equilibrium relative to a ``pure'' minimum disequilibrium algorithm with no warmstarts.  
}

\cmtt{
\textbf{Price-taker QCQP setting}. 
For the Adhya1-based instance, Algorithm~\ref{algo:cut_generating_minimum_disequilibrium} finds an equilibrium in one iteration given three of the five random seeds (seeds 2, 3, and 5).  The reason why the RGAP values are positive is because the RMP could not be solved to global optimality in 600 seconds (the time limit that we used throughout for the RMP) and thus the term $p_j(\v{x}_j^*,\vpi^*)$ appearing in the relative disequilibrium gap definition (Equation~\eqref{eq:relative_disequilibrium_gap_definition}) was imprecise.  Table~\ref{tbl:Complementarity_results} confirms that this is an equilibrium.
Similarly, for the Bental5-based instance, Algorithm~\ref{algo:cut_generating_minimum_disequilibrium} finds an equilibrium in the first major iteration for all random seeds. In contrast to the Adhya1-based instance, the vast majority of the solution time is spent proving optimality in player 2's subproblem, not in solving the RMP.  In particular, an equilibrium (optimal primal solution) is found within seconds, while the remaining time is spent computing $p_j^U(\vpi^*)$.  But because Gurobi 9.5 could not solve player 2's subproblem in the time allotted, the algorithm returned $p_j^U(\vpi^*) > p_j^*(\vpi^*) = p_j(\v{x}_j^*,\vpi^*)$, leading to an RGAP $\delta_j^*\%$ value above 1\%.  We believe that this could be improved by using a stronger pooling formulation. Recall that our implementation relies on the $\mathbb{P}$ formulation, which is known to furnish an inferior relaxation relative to the $\mathbb{PQ}$ formulation.
}

\cmtt{
\textbf{Price-taker MIQCQP setting}. In all ``Bental5 PT 1'' rows, the minimum disequilibrium algorithm terminates with a lower bound $\eta^{LB} > 0$ signifying that no equilibrium exists.  Random seed 5 required two major iterations of Algorithm~\ref{algo:cut_generating_minimum_disequilibrium}, while all other seeds required one iteration. Note that the upper and lower bounds did not yet match, but this is permissible since the algorithm can terminate as soon as the lower bound $\eta^{LB}$ exceeds 0.
Table~\ref{tbl:Complementarity_results} shows that the complementarity heuristic returned a solution and could not identify that no equilibrium exists, just as we encountered in the Haverly-esque price-taker MIQCQP instance. Meanwhile, for the Adhya1-based instance, Algorithm~\ref{algo:cut_generating_minimum_disequilibrium} finds an equilibrium and proves that it is an equilibrium in the first major iteration for all random seeds in roughly 20 seconds. 
}

\cmtt{
\textbf{Nash-Cournot QCQP setting}. Algorithm~\ref{algo:warmstart_method_Nash_Cournot_Gauss_Seidel_type} succeeded in finding an equilibrium on the first try for the Adhya1-based instance, while an equilibrium could not be proven in the Bental5-based instance.  For the latter, we report results corresponding to eight different random seeds to demonstrate how longer solve times yield a better (smaller) lower bound $\eta^{LB}$, a lower upper bound $\eta^{UB}$, and lower $\textrm{RGAP } \delta_j^*$ values.  
}

\cmtt{
\textbf{Nash-Cournot MIQCQP setting}. For both the Adhya1- and Bental5-based Nash-Cournot MIQCQP instances, the warmstart heuristic (Algorithm~\ref{algo:warmstart_method_Nash_Cournot_Gauss_Seidel_type}) was able to immediately find an equilibrium, which is apparent in the solution time and the fact that only one ``WS Cut'' was generated per player. 
For the Bental5-based instance, the $p_j^*$ columns reveal that there are multiple equilibria for this instance as each random seed leads to a distinct equilibrium. 
We also show results associated with a ``pure'' version of Algorithm~\ref{algo:cut_generating_minimum_disequilibrium} (``pure'' means no warmstart phase was applied, hence $\mc{X}_j^L = \emptyset$ for $j \in \mc{J}$). Without initial solutions, the algorithm is completely dependent on the RMP for generating candidate equilibria. 
The last row of Table~\ref{tbl:Adhya1_results} shows how a ``pure'' version of Algorithm~\ref{algo:cut_generating_minimum_disequilibrium} requires a much longer solve time to generate a near equilibrium. Interestingly, the ``pure'' Algorithm~\ref{algo:cut_generating_minimum_disequilibrium} implementation finds a different equilibrium than the practitioner's method (Algorithm~\ref{algo:warmstart_method_Nash_Cournot_Gauss_Seidel_type}). 
The last two rows of Table~\ref{tbl:Bental5_results} show once again that a ``pure'' version of Algorithm~\ref{algo:cut_generating_minimum_disequilibrium}, given a maximum limit of 20 major iterations, requires longer solve times. In the last row, the RMP is solved to optimality in each iteration, whereas in the second to last row, the RMP is given a time limit 10 seconds in each iteration.  
}

\cmtt{
\subsubsection{Complementarity heuristic performance} \label{sec:algorithm_scaleup_complementarity_performance}
}
\cmtt{
It is instructive to compare our minimum disequilibrium algorithm against the complementarity heuristic.
Table~\ref{tbl:Complementarity_results} shows the results of the complementarity heuristic under a variety of parameter settings.
Because the complementarity heuristic is not guaranteed to return an equilibrium, we must add a verification post-processing step to confirm that the returned solution is, indeed, an equilibrium.  Recall that the minimum disequilibrium algorithm includes this verification and the solution times reported reflect this time.
In the ``Solver(s)'' column, ``CONOPT 4.1'' means that CONOPT 4.1 was invoked with an initial solution $(\v{x}_1,\v{x}_2)=(\v{0},\v{0})$ and then terminated with a locally optimal solution. In all instances, CONOPT 4.1 terminates in under 0.1 seconds.
``Gurobi 9.5'' means that Gurobi 9.5 was used as an exact solver for the time appearing in the column ``Comp Model.''
``Gurobi 9.5 $\rightarrow$ CONOPT 4.1'' means that Gurobi 9.5 was invoked for the time appearing in the column ``Comp Model'' to generate an initial primal solution to the complementarity model (Model~\eqref{model:max_social_welfare_projected} for the price-taker setting and Model~\eqref{model:monolithic_model_for_nonconvex_example}, possibly with integer decisions, for the Nash-Cournot setting), which was then passed to the local solver CONOPT 4.1 to identify a local maximum.  For example, the second row shows that Gurobi 9.5 was given 5 seconds to generate an initial solution to the complementarity model, which was then polished by CONOPT 4.1.  
Note that we could not use CONOPT 4.1 in the MIQCQP setting because it cannot handle integer decision variables.
Finally, since multiple equilibria may exist, one must be careful comparing different methods.
}

\cmtt{
We first comment on several general trends that are apparent over many complementarity instances.
QCQP instances are more challenging than MIQCQP instances either because it is difficult to find an equilibrium (i.e., solve the QCQP complementarity model) or to verify that the solution returned is, in fact, an equilibrium (i.e., solve the QCQP subproblem to global optimality).  
Whenever CONOPT 4.1 was initialized with the zero solution $(\v{x}_1,\v{x}_2)=(\v{0},\v{0})$, it failed to identify an equilibrium and, in fact, fared quite poorly.  This observation re-iterates our assertion that local optimality of a complementarity model does not guarantee that an equilibrium has been found.  Gurobi 9.5 was able to solve all MIQCQP complementarity models quickly in under 20 seconds.  Other differences across the complementarity instances are highlighted below. 
}

\afterpage{%
    \clearpage
    \thispagestyle{empty}
    \begin{landscape}
\begin{table}																							
\centering	
{\color{black}																						
\begin{tabular}{ccclrrrrrrrr}																						
\toprule																							
Base Instance	&	Comp	&	Int	&	Solver(s)	&	\multicolumn{2}{c}{$p_j^*$}			&	\multicolumn{2}{c}{RGAP $\delta_j^*$\%}			&		\multicolumn{4}{c}{Solution Time [s]}					\\
								\cmidrule(lr){5-6}				\cmidrule(lr){7-8}						\cmidrule(lr){9-12}					
	&		&		&		&	1	&	2	&	1	&	2	&	Comp Model	&	1	&	2	&	Total	\\
\midrule																							
Adhya1	&	PT	&	0	&	CONOPT 4.1	&	0.00	&	0.00	&	100.00	&	100.00	&	0	&	2	&	1	&	2	\\
Adhya1	&	PT	&	0	&	Gurobi 9.5 $\rightarrow$ CONOPT 4.1	&	354.80	&	342.92	&	0.00	&	0.42	&	5	&	5	&	1	&	12	\\
Adhya1	&	PT	&	0	&	Gurobi 9.5	&	354.66	&	344.28	&	0.05	&	0.03	&	600	&	5	&	3	&	608	\\
Adhya1	&	PT	&	0	&	Gurobi 9.5 $\rightarrow$ CONOPT 4.1	&	354.80	&	344.37	&	0.00	&	0.00	&	600	&	5	&	1	&	607	\\
Adhya1	&	PT	&	1	&	Gurobi 9.5	&	87.50	&	97.50	&	0.00	&	0.00	&	5	&	3	&	0	&	8	\\
\\
Adhya1	&	NC	&	0	&	CONOPT 4.1	&	0.00	&	0.00	&	100.00	&	100.00	&	0	&	5	&	5	&	10	\\
Adhya1	&	NC	&	0	&	Gurobi 9.5	&	354.80	&	342.92	&	0.01	&	0.42	&	100	&	40	&	8	&	148	\\
Adhya1	&	NC	&	0	&	Gurobi 9.5 $\rightarrow$ CONOPT 4.1	&	354.80	&	342.92	&	0.05	&	0.42	&	600	&	60	&	20	&	680	\\
Adhya1	&	NC	&	1	&	Gurobi 9.5	&	88.06	&	96.39	&	0.00	&	0.00	&	16	&	0	&	0	&	16	\\
\midrule																							
Bental5	&	PT	&	0	&	CONOPT 4.1	&	100.00	&	100.00	&	98.31	&	98.28	&	0	&	0	&	5	&	5	\\
Bental5	&	PT	&	0	&	Gurobi 9.5	&	600.00	&	475.00	&	0.00	&	1.22	&	30	&	0	&	300	&	330	\\
Bental5	&	PT	&	1	&	Gurobi 9.5	&	200.00	&	-75.00	&	0.00	&	0/0	&	6	&	0	&	0	&	6	\\
\\
Bental5	&	NC	&	0	&	CONOPT 4.1	&	401.39	&	401.39	&	87.29	&	86.92	&	0	&	0	&	30	&	30	\\
Bental5	&	NC	&	0	&	Gurobi 9.5  &	1736.75	&	1546.40	&	0.00	&	0.52	&	5	&	1	&	30	&	36	\\
Bental5	&	NC	&	0	&	Gurobi 9.5 $\rightarrow$ CONOPT 4.1	&	1736.75	&	1546.40	&	0.00	&	0.58	&	5	&	1	&	30	&	36	\\
Bental5	&	NC	&	0	&	Gurobi 9.5 $\rightarrow$ CONOPT 4.1	&	1736.75	&	1546.40	&	0.00	&	0.43	&	5	&	1	&	180	&	186	\\
Bental5	&	NC	&	1	&	Gurobi 9.5	&	1374.11	&	1001.46	&	0.00	&	0.00	&	4	&	0	&	0	&	5	\\
\bottomrule																							
\end{tabular}																							
\caption{Complementarity heuristic results. ``CONOPT 4.1" = CONOPT 4.1 was invoked with an initial solution $(\v{x}_1,\v{x}_2)=(\v{0},\v{0})$.  This approach performed poorly in all cases.  ``Gurobi 9.5 $\rightarrow$ CONOPT 4.1" = Gurobi 9.5 was invoked with the time limit shown under ``Comp Model" to find an initial solution, which was then passed as an initial solution to CONOPT 4.1. ``Comp Model'' = Model~\eqref{model:max_social_welfare_projected} in the price-taker setting and Model~\eqref{model:relaxed_monolithic_NC} in the Nash-Cournot setting. The subheadings ``1'' and ``2'' under ``Solution Time [s]'' refer to the time spent solving player 1 and 2's subproblem, respectively, to confirm that an equilibrium has been found. All other column headers are explained in Table~\ref{tbl:Table_Key}. Note that ``Bental5 PT 1'' does \textit{not} return an equilibrium.}
\label{tbl:Complementarity_results}	
} 
\end{table}																						
\end{landscape}
    \clearpage
}																														

\cmtt{
\textbf{QCQP setting}. As mentioned above, the QCQP instances are challenging.
For the Adhya1-based instances, the difficulty is in finding an equilibrium (i.e., solving the complementarity model) as can be seen by noting (see Table~\ref{tbl:Adhya1_results}) that an equilibrium to both the price-taker and Nash-Cournot QCQP instances is when player 1 and 2 have a profit of 354.80 and 344.37, respectively.
In the price-taker setting (``Adhya1 PT 0''), three rows are shown with Gurobi 9.5 to demonstrate the performance of an exact solver. 
In the second row of Table~\ref{tbl:Complementarity_results}, using Gurobi 9.5 alone for only 5 seconds and then passing this solution to CONOPT 4.1 for polishing did not produce an equilibrium. This result re-iterates that a local solver is not guaranteed to find an equilibrium even if it is given an arguably better initial solution than the zero solution. 
In the third and fourth rows of Table~\ref{tbl:Complementarity_results}, we see that using Gurobi 9.5 alone for 600 seconds returned a near equilibrium, while using Gurobi 9.5 for 600 seconds and then polishing this solution with CONOPT 4.1 found a true equilibrium.  This equilibrium was confirmed in 6 seconds by having Gurobi 9.5 solve each subproblem to global optimality to corroborate that no player has an incentive to deviate.   
In the Nash-Cournot QCQP setting (``Adhya1 NC 0''), the exact solver Gurobi 9.5 returns an approximate equilibrium in 100 seconds (roughly the time required by our modified Practitioner's heuristic (Algorithm~\ref{algo:warmstart_method_Nash_Cournot_Gauss_Seidel_type})), but player 2's profit of 342.921 corresponds to a 0.42\% relative disequilibrium.
This is because Gurobi could not yet find a better primal solution in 100 seconds.
Even with a longer solve time for Gurobi (600 seconds) and polishing the solution with CONOPT, an equilibrium could not be found.
In contrast, for the Bental5-based instances, Gurobi 9.5 found an equilibrium in under 30 seconds in the price-taker setting and within 5 seconds in the Nash-Cournot setting (while attempting to \emph{verify} equilibrium dominated the total computational time). Note that the same difficulty that the minimum disequilibrium algorithm encountered when solving the subproblems to global optimality (see the third set of rows in Table~\ref{tbl:Bental5_results}) occurred here as well. Gurobi could not solve the $\mathbb{P}$ formulations to global optimality in three minutes.  
}

\cmtt{
\textbf{MIQCQP setting}. In the MIQCQP setting (``Adhya1 PT 1'' and ``Adhya1 NC 1''), the complementarity approach requires roughly 8 and 16 seconds, respectively, to find and confirm an equilibrium.
In contrast our minimum disequilibrium algorithm takes roughly 21 and 1 second(s), respectively, to find and confirm an equilibrium.  Admittedly, in the latter case, the modified Practitioner's heuristic (Algorithm~\ref{algo:warmstart_method_Nash_Cournot_Gauss_Seidel_type}) is responsible for the fast solve time.
On the other hand, our ``pure'' minimum disequilibrium method (Algorithm~\ref{algo:cut_generating_minimum_disequilibrium}), where ``pure'' means that no warmstart procedure was used (hence, $\mc{X}_j^L = \emptyset$ for all $j \in \mc{J}$), cannot assert that an equilibrium has been found for ``Adhya1 NC 1'' in less than 356 seconds.  
For the Bental5-based instances, the most salient observation is that the complementarity heuristic falsely identifies a solution as an equilibrium in the ``Bental5 PT 1'' row.
Specifically, we see that player 2's purported optimal profit is negative (-75).  
The complementarity heuristic has no difficulty with ``Bental5 NC 1'' as it solves in 5 seconds.
}

\cmtt{
In summary, if we compare the best performing version of each algorithm in terms of solution time, RGAP, and ability to find an equilibrium, by inspecting Tables~\ref{tbl:Adhya1_results}, \ref{tbl:Bental5_results}, and \ref{tbl:Complementarity_results}, we see that the minimum disequilibrium algorithm outperforms the complementarity heuristic on the Adhya1-based Nash-Cournot instances, while the opposite is true for the Adhya1-based price-taker instances. The complementarity approach also appears slightly better for the Bental5-based QCQP instances as the RGAPs appear lower in the time allotted.  Both methods are equally fast on the Bental5-based Nash-Cournot MIQCQP instance. Finally, the complementarity approach fails on the Bental5-based price-taker MIQCQP instance, while the minimum disequilibrium algorithm successfully identifies the non-existence of an equilibrium.  
If we look at inferior versions of each algorithm, then both can struggle.  The pure minimum disequilibrium algorithm with no warmstarts can take a long time to converge to a near equilibrium, while the complementarity heuristic, if given the wrong starting solution, can return low-quality solutions when solved to local optimality. Not surprisingly, there is no silver bullet, but there is clear empirical evidence that our minimum disequilibrium offers notable benefits in certain settings.
}

\section{Conclusions and Future Work} \label{sec:conclusions}

This work introduced ``competitive'' pooling problems as a paradigm for studying the interaction of multiple non-cooperative nonlinear players each seeking to maximize individual profit. Such problems are emblematic of many nonconvex process systems engineering-related problems in which multiple agents vie with one another.  Rigorously handling these nonconvexities has, to date, been quite challenging and has likely dissuaded researchers and practitioners from pursuing such problems.  In addition to showing the potential of complementarity-based heuristics, we applied an exact decomposition algorithm aimed at minimizing disequilibrium to solve these problems to provable global optimality.  Unlike existing methods, our decomposition algorithm produces a ``certificate of non-existence'' when no equilibrium exists. We believe the algorithm holds promise for tackling related problems in the optimization and process systems engineering community. 

\cmtt{
Algorithmically, several concluding remarks are in order. 
First, in the price-taker setting when no equilibrium exists, we observed that the complementarity approach was rather susceptible to fail (i.e., return a candidate equilibrium) in both the QCQP and MIQCQP settings. 
Second, our ``pure'' minimum disequilibrium method (Algorithm~\ref{algo:cut_generating_minimum_disequilibrium} without any warmstart solutions, i.e., $\mc{X}_j^L = \emptyset$ for all players $j \in \mc{J}$) was significantly slower than the practitioner's heuristic (Algorithm~\ref{algo:practitioners_method_Jacobi_type}) and complementarity-based heuristics (solving Model~\eqref{model:max_social_welfare_projected} or Model~\eqref{model:relaxed_monolithic_NC}) at finding equilibria.  In other words, our results corroborate what is already known -- that the latter two methods are good heuristics -- and demonstrate that our ``pure'' minimum disequilibrium method is inferior at finding good primal solutions. 
Third, since our minimum disequilibrium algorithm is an effective ``dual'' approach in that it provides equilibrium guarantees, it is our opinion that hybrid methods, e.g., algorithms incorporating good primal heuristics within our minimum disequilibrium framework, offer the community powerful and rigorous tools for solving nonconvex games. Coupling our approach with a single-iteration practitioner's heuristic proved to be competitive with existing methods. One could certainly explore a multi-iteration practitioner's heuristic or coupling a complementarity-based heuristic with our algorithm as well. 
}

There are numerous opportunities for additional research.  
First, whereas we only looked into static problems, many problems of practical interest require dynamic (time-indexed) settings in which decisions are connected over time. 
Second, consideration of stochastic parameters would open the doors to investigating equilibrium under uncertainty. Stochastic settings in which players optimize an expected profit function subject to global scenarios is already done with convex players, but is not well understood when nonconvexities are present. 
Third, it would be interesting to consider correlated and/or nonlinear inverse demand functions.  This variant would likely be even more complicated for complementarity-based approaches to find optimal solutions; we believe that our decomposition algorithms may be more promising in this setting. 
Fourth, treating a generalized Nash setting, in which the upper-level decision variables appear in the lower-level constraints, is also on our wish list for further study.  Generalized Nash instances arise when an individual player's feasible region (not just her objective function) is affected by other players' decisions.
Fifth, it is possible for multiple equilibria to exist and a tractable approach to identify all of them is not yet known.
Sixth, whereas we represented all player behavior as either ``price taker'' or ``price maker,'' it is also possible to model mixtures of player behavior with some players acting as price takers and other acting as price makers. 
Finally, there are other forms of equilibrium besides a Nash equilibrium.  Pursuing these other forms may lead to new, or at least different insights concerning the strategic behavior of multiple interacting players.      

\cmtt{
\section{Acknowledgments} 
}
\cmtt{
We thank Youngdae Kim for an insightful suggestion on how to best demonstrate when a complementarity approach may fail. We thank an astute referee for helping us showcase the generality of our decomposition framework on challenging multi-quality pooling instances. 
}

\small
\bibliographystyle{abbrv}
\bibliography{gma_refs,lotero_refs,gsip}

\appendix

\section{Perfect Competition: Alternative derivation of maximum social welfare model via complementarity}
\label{app:msw_via_comp}

This section provides a standalone derivation of the maximum social welfare model \eqref{model:max_social_welfare} using a complementarity framework. This derivation allows one to interpret the maximum social welfare model \eqref{model:max_social_welfare} through the lens of complementarity, which may offer interesting connections in its own right.

In the perfect competition setting, player $j$'s objective function~\eqref{objfnc:price_taker_indep} can be represented more generally and concisely as 
\begin{equation} \label{objfnc:perfect_comp_indep_linear_x}
p^{\textrm{SupplierPC}}_j(\v{x}_j,\vpi) = 
\sum_{r \in \mc{R}} \pi_r x_{jr} - \v{C}_j^{\top} \v{x}_j  
\end{equation}
where the set $\mc{R}$ indexes the output streams where revenue is earned, i.e., $\mc{R} = \Nout$, and $\v{C}_j$ denotes a vector of cost parameters.
This leads to a compact representation of player $j$'s \textit{relaxed} optimization problem, parameterized by the price vector $\vpi$:
\begin{equation} \label{model:relaxed_individual_player_PC}
\max_{\v{x}_j \in \mc{X}^C_j}~ \sum_{r \in \mc{R}} \pi_r x_{jr} - \v{C}_j^{\top} \v{x}_j 
\end{equation}
Assuming a set of linear inverse demand functions, the consumer attempts to maximize consumer surplus by solving 
\begin{equation} \label{model:market_player_PC_linear_demand}
\max_{\v{q}}~ p^{\textrm{ConsumerPC}}_j(\v{q}) = \sum_{r \in \mc{R}} \alpha_r q_r - \tfrac{1}{2}\beta_r q_r^2 - \pi_r q_r. 
\end{equation}
Finally, we need to enforce that supply and demand are equal
\begin{equation} \label{eq:supply_demand_balance}
q_r = \sum_{j \in \mc{J}} x_{jr} \qquad \forall r \in \mc{R}.
\end{equation} 

\begin{proposition}
The KKT conditions of the continuous relaxation of Model~\eqref{model:max_social_welfare}
are identical to the complementarity system obtained from aggregating (i) the KKT conditions of Model~\eqref{model:relaxed_individual_player_PC} for all players $j \in \mc{J}$, 
(ii) the KKT conditions of Model~\eqref{model:market_player_PC_linear_demand},
and (iii) supply-demand balance equations~\eqref{eq:supply_demand_balance}. 
\end{proposition}
\proof 
We can express the \textit{relaxed} maximum social welfare Model~\eqref{model:max_social_welfare} as
\begin{alignat}{3} \label{model:max_social_welfare_generic}
\max_{\v{q},\v{x}}~ & p^{\textrm{RMaxSW}}(\v{q},\v{x}) = \sum_{r \in \mc{R}}(\alpha_{r} q_{r} - \tfrac{1}{2} \beta_{r} q_{r}^2) - \sum_{j \in \mc{J}} \v{C}_j^{\top} \v{x}_j \notag \\
\st~~ 				
	   & \v{x}_j \in \mc{X}^C_j 	\qquad \forall j \in \mc{J} \\
\notag & q_r = \sum_{j \in \mc{J}} x_{jr}    \quad \forall r \in \mc{R} \qquad (\textrm{dual vars: } \pi_r \in \Re)
\end{alignat}
where we let $\pi_r$ denote the dual variables associated with the supply-demand balance equations.
Let $L^{\textrm{SupplierPC}}_j$, $L^{\textrm{ConsumerPC}}$, and $L^{\textrm{RMaxSW}}$ denote the Lagrangian of \eqref{model:relaxed_individual_player_PC}, \eqref{model:market_player_PC_linear_demand}, and \eqref{model:max_social_welfare_generic}, respectively.

The KKT conditions of Model~\eqref{model:max_social_welfare_generic} are
\begin{subequations} \label{model:aa}
\begin{alignat}{4}
\frac{\partial L^{\textrm{RMaxSW}}(\v{q},\v{x})}{\partial x_{jr}} = 
\frac{\partial L^{\textrm{SupplierPC}}_j(\v{x}_j,\vpi)}{\partial x_{jr}} = \qquad \qquad \qquad & & \notag \\
\qquad \pi_r - C_{jr} 
+ \vlambda_j^{\top}\nabla_{x_{jr}} \v{g}_j(\v{x}_j) 
+ \vmu_j^{\top}\nabla_{x_{jr}} \v{h}_j(\v{x}_j)
- \nu_{jr} = 0 & & ~~\forall j \in \mc{J},r \in \mc{R} \label{eq:RelaxedMaxSW_partial_xjr} \\
\frac{\partial L^{\textrm{RMaxSW}}(\v{q},\v{x})}{\partial x_{js}} = 
\frac{\partial L^{\textrm{SupplierPC}}_j(\v{x}_j,\vpi)}{\partial x_{js}} = \qquad \qquad \qquad & & \notag \\
\qquad - C_{js} 
+ \vlambda_j^{\top}\nabla_{x_{js}} \v{g}_j(\v{x}_j) 
+ \vmu_j^{\top}\nabla_{x_{js}} \v{h}_j(\v{x}_j)
- \nu_{js} = 0 & & ~~\forall j \in \mc{J},s \notin \mc{R} \label{eq:RelaxedMaxSW_partial_xjs} \\
\frac{\partial L^{\textrm{RMaxSW}}(\v{q},\v{x})}{\partial q_{r}} = 
\frac{\partial L^{\textrm{ConsumerPC}}(\v{q})}{\partial q_{r}} = 
\alpha_r - \beta_r q_r - \pi_r = 0 & & ~~\forall r \in \mc{R} \label{eq:RelaxedMaxSW_partial_qr} \\
\eqref{eq:gj_perp_lambdaj},\eqref{eq:hj_perp_muj},\eqref{eq:xj_perp_nuj} & & ~~\forall j \in \mc{J} \label{eq:RelaxedMaxSW_complementarity_constraints} \\
q_r = \sum_{j \in \mc{J}} x_{jr}, \pi_r \in \Re & &~~\forall r \in \mc{R} \label{eq:RelaxedMaxSW_supply_demand_balance},
\end{alignat}
\end{subequations}
Meanwhile, the KKT conditions for player $j \in \mc{J}$ are identical to the complementarity system~\eqref{model:generic_nlp_kkt} and re-expressed in \eqref{eq:RelaxedMaxSW_partial_xjr}, \eqref{eq:RelaxedMaxSW_partial_xjs}, and \eqref{eq:RelaxedMaxSW_complementarity_constraints}.
The KKT conditions of Model~\eqref{model:market_player_PC_linear_demand} are given in \eqref{eq:RelaxedMaxSW_partial_qr}. Finally, supply-demand balance constraints are captured in \eqref{eq:RelaxedMaxSW_supply_demand_balance}.
\qed

\section{\cmtt{An example where global solution of a monolithic reformulation of a complementarity model fails}}
\label{app:counterexample}

\cmtt{
We have seen throughout our examples and numerical experiments a number of situations where the complementarity heuristics may fail:
From Section~\ref{sec:kkt_fails}, when an equilibrium exists, solving a complementarity problem might not yield an equilibrium
(since the solution set of the complementarity problem is a superset of the set of equilibria);
From Section~\ref{sec:comp_fails} and Table~\ref{tbl:Complementarity_results}	, when an equilibrium exists, solving a complementarity problem via a local solution of a monolithic reformulation failed to produce an equilibrium;
and from Section~\ref{sec:no_equilibrium}, when an equilibrium did not exist, solving a monolithic reformulation produced a candidate solution that, obviously, could not be an equilibrium.
}
\cmtt{However, we did see that \emph{global} solution of a monolithic reformulation of a complementarity problem was reasonably effective as a heuristic for finding an equilibrium when one exists.
The following example shows that, in general, we cannot rely on this property.
}

Consider a classical continuous two-player Nash-Cournot game where player $j \in \{1,2\}$ solves  
\begin{equation}
\tilde{p}_j(x_{-j}) = \max \{ -x_j x_{-j} : x_j \in \Re \} = \min \{ x_j x_{-j} : x_j \in \Re \}.
\end{equation}
This game has a unique equilibrium at $x_1 = x_2 = 0$. 
Note that player $j$'s objective function $x_j x_{-j}$ can be viewed as a special case of a more general objective function involving a linear inverse demand function $(\alpha - \beta(x_j + x_{-j}))x_j$ with $\beta=1$ and a cost term of $x_j^2 - \alpha x_j$ where $\alpha$ is sufficiently large such that $\alpha x_j$ dominates $x_j^2$.
The first-order KKT condition for player $j$ is
\begin{equation} \label{model:player_j_kkt_example_comp_fails}
x_{-j} = 0.
\end{equation}
As a consequence, combining the first-order KKT conditions for both players yields the unique equilibrium $x_1 = x_2 = 0$. 
However, the ``monolithic'' nonconvex QP optimization problem
\begin{equation} \label{model:monolithic_model_for_nonconvex_example}
\min \big\{ x_1 x_2 : x_j \in \Re, j \in \{1,2\} \big\}
\end{equation}
gives rise to the same KKT conditions and could therefore be used/solved to find equilibria.
One can see by inspection that no optimum exists to problem \eqref{model:monolithic_model_for_nonconvex_example} as it is clearly unbounded.
This example illustrates that solving a ``monolithic'' nonconvex QP to global optimality in order to hunt for equilibria may fail.

\end{document}